\documentclass[journal,final]{IEEEtran}
\usepackage{hyperref}
\usepackage{array}
\usepackage{dblfloatfix}
\usepackage{url}
\usepackage{fixltx2e}
\usepackage{enumitem}
\usepackage[utf8]{inputenc}
\usepackage{bbm}
\usepackage[makeroom]{cancel}
\usepackage{soul}
\usepackage{upgreek}
\usepackage[colorinlistoftodos]{todonotes}
\usepackage{listings}
\usepackage{cuted}
\usepackage[noadjust]{cite}
\usepackage{amsthm}
\usepackage{amsmath,amssymb,amsfonts,nccmath,mathtools}

\usepackage{accents}	
\makeatletter	
\newcommand{\sbullet}{%
\hbox{\fontfamily{lmr}\fontsize{.4\dimexpr(\f@size pt)}{0}\selectfont\textbullet}}

\usepackage{graphicx}
\usepackage{textcomp}
\usepackage{xcolor}
\usepackage{color}
\usepackage{accents}
\usepackage{tikz}
\usetikzlibrary{positioning}
\usepackage{float,layouts}
\usepackage[justification=centering]{caption}
\usepackage{cuted} 
\usepackage{ulem}

\usepackage{algorithm,algorithmic}

\usepackage[linesnumbered,ruled,vlined,algo2e]{algorithm2e}

\SetCommentSty{mycommfont}
\let\oldnl\nl
\newcommand{\nonl}{\renewcommand{\nl}{\let\nl\oldnl}}
\SetKwInput{KwInput}{Input}                
\SetKwInput{KwOutput}{Output}              
\SetKwInOut{Parameter}{Parameter}
\SetKwRepeat{Do}{do}{while}%

\def\BibTeX{{\rm B\kern-.05em{\sc i\kern-.025em b}\kern-.08em
		T\kern-.1667em\lower.7ex\hbox{E}\kern-.125emX}}

\usepackage{mathrsfs}
\usepackage{epstopdf}
\usepackage{amsfonts}
\usepackage{epsf,epsfig,verbatim,multicol,multirow}  
\usepackage{psfrag,bm,xspace}
\usepackage{hhline}
\usepackage{physics}
\newtheorem{thm}{Theorem}
\newtheorem{prop}{Proposition}
\newtheorem{dfn}{Definition}
\newtheorem{rem}{Remark}
\newtheorem{lem}[thm]{Lemma}
\newtheorem{cor}{Corollary}[thm]

\newtheorem{assumption}{Assumption}

\newcounter{relctr} 
\everydisplay\expandafter{\the\everydisplay\setcounter{relctr}{0}} 

\newcommand\labelrel[2]{%
  \begingroup
    \refstepcounter{relctr}%
    \stackrel{\textnormal{(\alph{relctr})}}{\mathstrut{#1}}%
    \originallabel{#2}%
  \endgroup
}
\AtBeginDocument{\let\originallabel\label} 

\allowdisplaybreaks 
\global\long\def\11{\mathbbm{1}}

\newcommand\numberthis{\addtocounter{equation}{1}\tag{\theequation}}
\newcommand{\ra}{\rightarrow}

\newcommand{\mcl}{\mathcal}
\newcommand{\mbb}{\mathbb}
\newcommand{\mbf}{\mathbf}
\newcommand{\eps}{\epsilon}
\newcommand{\ov}{\overline}
\newcommand{\udl}{\underline}
\newcommand{\uda}{\underaccent}
\newcommand{\ulbar}[1]{ \uda{\bar}{\bar{#1}} }

\def \defeq{\overset{\Delta}{=}}

\def \wh{\widehat}

\def \pr{\mbb{P}}

\def \l{\left}
\def \r{\right}

\newcommand{\macpomdp}{\text{MA-C-POMDP}}

\newcommand{\dotp}[2]{\langle #1, #2 \rangle}
\newcommand{\Stt}[1]{S_{#1}}
\newcommand{\stt}[1]{s_{#1}}
\newcommand{\sspace}{\mcl{S}}
\newcommand{\Otn}[2]{O_{#1}^{(#2)}}
\newcommand{\Ot}[1]{O_{#1}}

\newcommand{\ot}[1]{o_{#1}}

\newcommand{\ospace}{\mcl{O}}
\newcommand{\onspace}[1]{ \ospace^{(#1)} }
\newcommand{\Atn}[2]{A_{#1}^{(#2)}}
\newcommand{\At}[1]{A_{#1}}

\newcommand{\atn}[2]{a_{#1}^{(#2)}}
\newcommand{\at}[1]{a_{#1}}
\newcommand{\an}[1]{a^{(#1)}}
\newcommand{\aspace}{\mcl{A}}
\newcommand{\anspace}[1]{\aspace^{(#1)}}
\newcommand{\borel}[1]{\mcl{B}\l(#1\r)}

\newcommand{\m}[1]{\mcl{M}_1\l( #1 \r)}
\newcommand{\metric}[1]{d_{#1}}

\newcommand{\Hst}[1]{H_{#1}}
\newcommand{\Hstn}[2]{H_{#1}^{(#2)}}
\newcommand{\hst}[2]{#1_{#2}}

\newcommand{\hstn}[3]{#1_{#2}^{(#3)}}
\newcommand{\hsspace}{\mcl{H}}
\newcommand{\hstspace}[1]{\hsspace_{#1}}
\newcommand{\hsnspace}[1]{\hsspace^{(#1)}}
\newcommand{\hstnspace}[2]{\hsspace_{#1}^{(#2)}}
\newcommand{\policy}{u}
\newcommand{\uspace}{\mcl{U}}
\newcommand{\uspacen}[1]{\uspace^{(#1)}}
\newcommand{\mun}[1]{\mu^{(#1)}}
\newcommand{\uspacemix}{\uspace_{\mathrm{mixed}}}

\newcommand{\ut}[2]{#1_{#2}}
\newcommand{\un}[2]{#1^{(#2)}}
\newcommand{\utn}[3]{#1_{#2}^{(#3)}}
\newcommand{\prup}[2]{ \pr^{\l(#1\r)}_{#2} }
\newcommand{\E}[2]{ \mbb{E}^{\l(#1\r)}_{#2}}
\newcommand{\cCost}{c\l( \Stt{t}, \At{t} \r)}

\newcommand{\dCost}{ d\l( \Stt{t}, \At{t} \r) }

\newcommand{\constraintv}{\breve{D}}
\newcommand{\lag}[2]{L^{(P_1, \alpha)} \l(#1, #2 \r)}
\newcommand{\lags}[2]{L \l(#1, #2 \r)}
\newcommand{\lagmix}[2]{ \wh{L}^{(P_1, \alpha)}\l(#1, #2 \r) }
\newcommand{\lagsmix}[2]{\wh{L}\l(#1, #2 \r) }

\newcommand{\fullccost}[1]{C^{(P_1, \alpha)}\l(#1\r)}
\newcommand{\fullccostmix}[1]{\wh{C}^{(P_1, \alpha)}\l(#1\r)}
\newcommand{\fullccosts}[1]{C\l(#1\r)}
\newcommand{\fullccostsmix}[1]{\wh{C}\l(#1\r)}

\newcommand{\fulldkcostsmix}[2]{\wh{D}_{#1}\l(#2 \r)}

\newcommand{\fulldcost}[1]{D^{(P_1, \alpha)}\l( #1 \r)}
\newcommand{\fulldcostmix}[1]{\wh{D}^{(P_1, \alpha)}\l( #1 \r)}
\newcommand{\fulldcosts}[1]{D\l( #1 \r)}
\newcommand{\fulldcostsmix}[1]{\wh{D}\l( #1 \r)}

\newcommand{\optcost}{\udl{C}^{(P_1, \alpha)}}
\newcommand{\optcosts}{\udl{C}}

\newcommand{\infsup}[2]{\inf\limits_{#1}\sup\limits_{#2}}
\newcommand{\supinf}[2]{\sup\limits_{#1}\inf\limits_{#2}}

\newcommand{\xuspace}{\mcl{X}_{\mcl{U}}}
\newcommand{\xuspacen}[1]{\mcl{X}_{\mcl{U}^{(#1)}}}

\newcommand{\useq}[2]{\hspace{0pt}^{#1}#2}
\newcommand{\pruphst}[3]{p_{P_1} \l(#1, #2, #3\r)}
\newcommand{\pruphsts}[3]{p\l(#1, #2, #3\r)}
\newcommand{\zuphst}[3]{W_{P_1}\l(#1, #2, #3\r)}
\newcommand{\zuphsts}[3]{W\l(#1, #2, #3\r)}

\graphicspath{{./figures/}}

\usepackage[T1]{fontenc}

\usepackage{cite}
\usepackage{hyperref}
\usepackage{amsmath}
\usepackage{algorithm,algorithmic}

\usepackage{array}

\ifCLASSOPTIONcompsoc
  \usepackage[caption=false,font=normalsize,labelfont=sf,textfont=sf]{subfig}
\else
  \usepackage[caption=false,font=footnotesize]{subfig}
\fi
\usepackage{dblfloatfix}
\usepackage{url}

\allowdisplaybreaks

\begin{document}
%
\title{
A Strong Duality Result for Constrained POMDPs with Multiple Cooperative Agents
}
%
%
%
%
\author{Nouman~Khan,~\IEEEmembership{Member,~IEEE,}
       and~Vijay~Subramanian,~\IEEEmembership{Senior Member,~IEEE}
}

\maketitle
\begin{abstract}
The work studies the problem of decentralized constrained POMDPs in a team-setting where multiple non-strategic agents have asymmetric information. Using an extension of Sion's Minimax theorem for functions with positive infinity and results on weak-convergence of measures, strong duality is established for the setting of infinite-horizon expected total discounted costs when the observations lie in a countable space, the actions are chosen from a finite space, the immediate constraint costs are bounded, and the immediate objective cost is bounded from below. 
\end{abstract}


%
\IEEEpeerreviewmaketitle

\section{Introduction}\label{sec:introduction}
\IEEEPARstart{S}{ingle-Agent} 
Markov Decision Processes (SA-MDPs) \cite{bellman57} and Single-Agent Partially Observable Markov Decision Processes (SA-POMDPs) \cite{astrom65} have long served as the basic building-blocks in the study of sequential decision-making. An SA-MDP is an abstraction in which an agent sequentially interacts with a fully-observable Markovian environment to solve a multi-period optimization problem; in contrast, in SA-POMDP, the agent only gets to observe a noisy or incomplete version of the environment. 
In 1957, Bellman proposed dynamic-programming as an approach to solve SA-MDPs \cite{bellman57,howard:dp}. This combined with the characterization of SA-POMDP into an equivalent SA-MDP \cite{smallwood1973optimal, sondik1978optimal, kaelbing199899} (in which the agent maintains a belief about the environment's true state) made it possible to extend dynamic-programming results to SA-POMDPs. 
\emph{Reinforcement learning}~\cite{sutton98} based algorithmic frameworks 
use data-driven dynamic-programming approaches to solve such single-agent sequential decision-making problems when the environment is unknown.

In many engineering systems, there are multiple decision-makers that collectively solve a sequential decision-making problem but with safety constraints: e.g., a team of robots performing a joint task, a fleet of automated cars navigating a city, multiple traffic-light controllers in a city, etc. Bandwidth constrained communications and communication delays in such systems lead to a decentralized team problem with information asymmetry. In this work, we study a fairly general abstraction of such systems, 
namely 
that of a cooperative multi-agent constrained POMDP, hereon referred to as \macpomdp. The special cases of \macpomdp 
when there are no constraints, when there is only one agent, or when the environment is fully observable to each agent, are referred to as MA-POMDP, SA-C-POMDP, and MA-C-MDP, respectively. The relationships among such models are shown in Figure \ref{fig:model_relationships}.

\begin{rem} 
MA-C-POMDP, is an extension of the decentralized POMDP (Dec-POMDP) to the setting of constrained decision-making, i.e., Decentralized Constrained POMDP (Dec-C-POMDP). Importantly, in this paper, MA-POMDP is equivalent to Dec-POMDP and MA-C-MDP to Dec-C-MDP. In Dec-POMDP or Dec-C-POMDP, agents are assumed to act based on their individual information without any communication with each other. 

For a good introduction to Dec-POMDPs, please see \cite{oliehoek16}. We inform beforehand that \cite{oliehoek16} considers MA-POMDP as a special case of Dec-POMDP wherein agents communicate all their information with each other. We have decided to deviate from this categorization because the term multi-agent itself does not specify whether agents engage in communication and/or the degree to which they do so.\footnote{Settings that involve communication can be incorporated in our MA-C-POMDP formulation through actions and observations of the agents (see~\cite{oliehoek16}).}
\end{rem}

\subsection{Related Work}
\subsubsection{Single-Agent Settings}
Prior work on planning and learning under constraints has primarily focused on single-agent constrained MDP (SA-C-MDP) where unlike in SA-MDPs, the agent solves a constrained optimization problem. For this setup, a number of fundamental results from the planning perspective have been derived -- for instance, \cite{altman94, altman96, feinberg94, feinberg95, feinberg96, feinberg2000, feinberg2020}; see \cite{altman-constrainedMDP} for details of the convex-analytic approach for SA-C-MDPs. 
These aforementioned results have led to the development of many algorithms in the learning setting: see \cite{borkar2005AnAA, bhatnagar2010AnAA, bhatnagar2012AnOA, Wei2022APM, wei22a-pmlr-v151, bura2021, vaswani2022}. Unlike SA-C-MDPs, rigorous results for SA-C-POMDPs are limited; few works include \cite{dongho2011, jongmin18, undurti2010, jamgochian2022}.

\subsubsection{Multi-Agent Settings}
Challenges arising from the combination of \emph{partial observability} of the environment and \emph{information-asymmetry}\footnote{
Mismatch in the information of the agents.} have led to difficulties in developing general solutions to MA-POMDPs: e.g., solving a finite-horizon MA-POMDP with more than two agents is known to be NEXP-complete~\cite{bernstein00}. Nevertheless, conceptual approaches exist to establish solution methodologies and structural properties in (finite-horizon) MA-POMDPs namely: i) the person-by-person approach~\cite{witsenhausen1979structure}; ii) the designer's approach~\cite{witsenhausen1973standard}; and iii) the \emph{common-information (CI) approach}~\cite{nayyar13,nayyar14}.
Using a fictitious \emph{coordinator} that only uses the common information to take actions, the CI approach allows for the transformation of the problem to a SA-POMDP which can be used to solve for an optimal control. 
The CI approach has also led to the development of a multi-agent reinforcement learning (MARL) framework \cite{hsu22} where agents learn good compressions of common and private information that can suffice for approximate optimality. On the empirical front, worth-mentioning works include 
\cite{gupta17,rashid2020-2}. 
Finally, as far as we know, work on 
MA-C-POMDPs 
is non-existent.

\subsection{Contribution}
For MA-C-POMDPs, the technical challenges increase even more from those of MA-POMDPs because restriction of the search space to deterministic policy-profiles is no longer an option\footnote{Restricting to deterministic policies can be sub-optimal in SA-C-MDPs and SA-C-POMDPs: see \cite{altman-constrainedMDP} and \cite{dongho2011}.}. Therefore, the coordinator in the equivalent SA-C-POMDP has an uncountable prescription space, which leads to an uncountable state-space in its equivalent SA-C-MDP. This is an issue because most fundamental results in the theory of SA-C-MDPs (largely based on occupation-measures) rely heavily on the state-space being at most countably-infinite; see \cite{altman-constrainedMDP}. Due to these reasons, the study of MA-C-POMDPs 
calls for a new methodology---one which avoids this transformation and directly studies the decentralized problem. Our work takes the first steps in this direction and presents a rigorous approach for MA-C-POMDPs 
which is based on structural characterization of the set of behavioral policies and their performance measures, and using measure theoretic results. The main result in this paper, namely Theorem \ref{thm:strongduality}, establishes strong duality and existence of a saddle-point for MA-C-POMDPs, thus providing a firm theoretical basis for (future) development of primal-dual type planning and learning algorithms.

\subsection{Organization}
The rest of the paper is organized as follows. Mathematical model of (cooperative) MA-C-POMDP is introduced in Section \ref{sec:problem}. The optimization problem is formulated in Section \ref{sec:optimization_problem}. Results on strong duality and existence of a saddle point are then derived in Section \ref{sec:strongduality}. 
Finally, concluding remarks are given in Section \ref{sec:conclusion}.

\subsection{Notation}
Before we present the model, we highlight the key notations in this paper.
\begin{itemize}[leftmargin=0pt, 
itemindent=10pt,
labelwidth=0pt, 
labelsep=5pt, 
]
\item The sets of integers and positive integers are respectively denoted by $\mbb{Z}$ and $\mbb{N}$. For integers $a$ and $b$, $[a,b]_{\mbb{Z}}$ represents the set $\{a, a+1, \dots, b\}$ if $a\le b$ and $\emptyset$ otherwise. The notations [a] and $[a,\infty]_{\mbb{Z}}$ are used as shorthand for $[1, a]_{\mbb{Z}}$ and $\{a, a+1, \dots \}$, respectively.
\item For integers $a \le b$ and $c \le d$, and a quantity of interest $q$, $\un{q}{a:b}$ is a shorthand for the vector $\l( \un{q}{a}, \un{q}{a+1}, \dots, \un{q}{b} \r)$ while $\ut{q}{c:d}$ is a shorthand for the vector $\l( \ut{q}{c}, \ut{q}{c+1}, \dots, \ut{q}{d} \r)$. The combined notation $\utn{q}{a:b}{c:d}$ is a shorthand for the vector $(\utn{q}{i}{j}: i \in [a,b]_{\mbb{Z}}, j \in [c, d]_{\mbb{Z}})$. The infinite tuples $\l( \un{q}{a}, \un{q}{a+1}, \dots, \r)$ and $\l( \ut{q}{c}, \ut{q}{c+1}, \dots, \r)$ are respectively denoted by $ \un{q}{a:\infty}$ and $\ut{q}{c:\infty}$.
\item For two real-valued vectors $v_1$ and $v_2$, the inequalities $v_1 \le v_2$ and $v_1 < v_2$ are meant to be element-wise inequalities.
\item Probability and expectation operators are denoted by $\pr$ and $\mbb{E}$, respectively. Random variables are denoted by upper-case letters and their realizations by the corresponding lower-case letters.  At times, we also use the shorthand  $\mbb{E}\l[ \cdot | x \r] \defeq \mbb{E}\l[ \cdot | X = x \r]$ and $\pr\l( y | x \r) \defeq \pr\l( Y = y | X = x \r)$ for conditional quantities.
\item Topological spaces are denoted by upper-case calligraphic letters. For a topological-space $\mcl{W}$, $\borel{\mcl{W}}$ denotes the Borel $\sigma$-algebra, measurability is determined with respect to $\borel{\mcl{W}}$, and $\m{\mcl{W}}$ denotes the set of all probability measures on $\borel{\mcl{W}}$ endowed with the topology of weak convergence. Also, unless stated otherwise, ``measure'' means a non-negative measure.
\item Unless otherwise stated, if a set $\mcl{W}$ is countable, as a topological space it will be assumed to have the discrete topology. Therefore, the corresponding Borel $\sigma$-algebra $\borel{\mcl{W}}$ will be the power-set $2^{\mcl{W}}$.
\item Unless stated otherwise, the product of a collection of topological spaces will be assumed to have the product topology.
\item The notation in Appendices \ref{sec:appendix:helpful_facts} and \ref{sec:appendix:minimax} is exclusive and should be read independent of the rest of the manuscript.
\end{itemize}
\begin{figure}[t]
    \centering
    \includegraphics[width=0.9\linewidth]{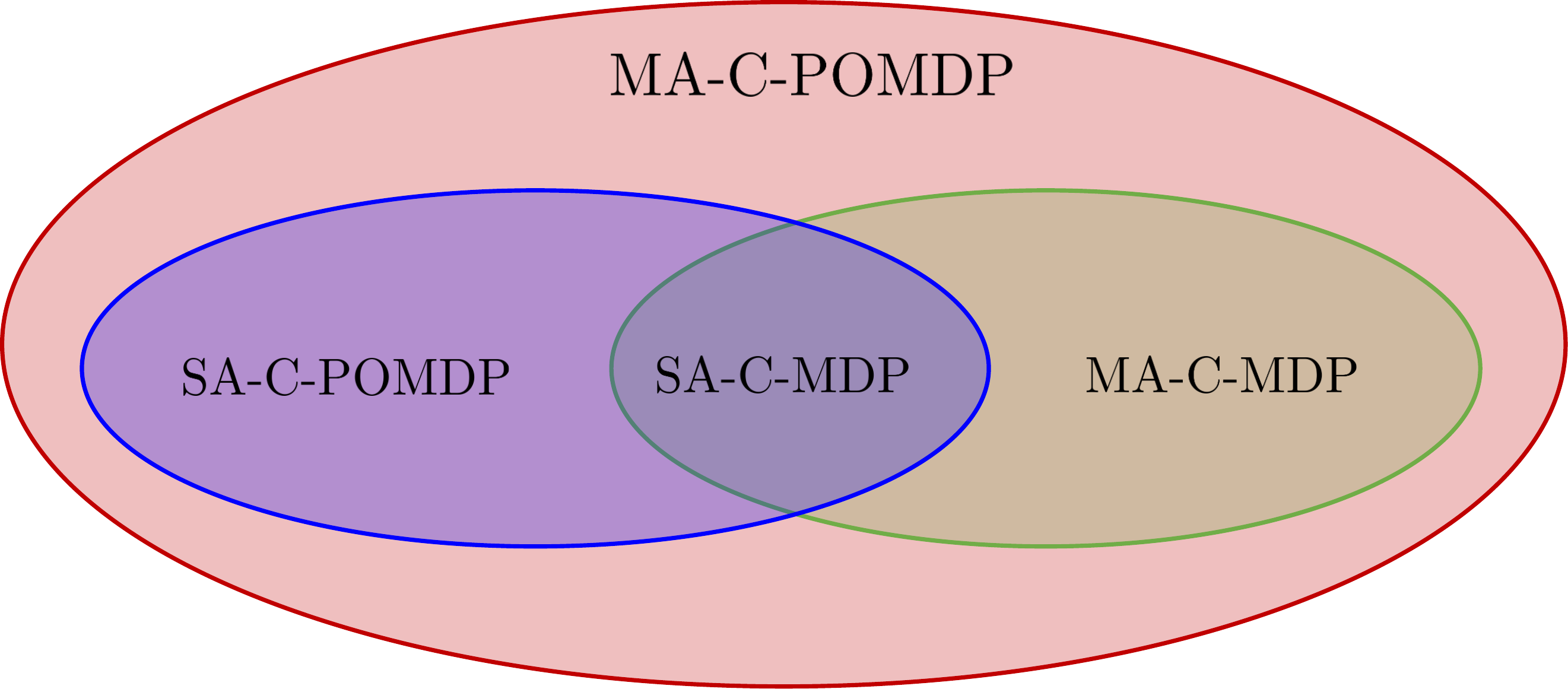}
    \caption{Relationships between Models of Cooperative Sequential Decision-Making under Constraints.}
    \label{fig:model_relationships}
\end{figure}
\section{Model}\label{sec:problem}
Let $ \l( N, \sspace, \ospace, \aspace, \mcl{P}_{tr}, \l( c, d\r), P_{1}, \mcl{U}, \alpha \r)$
denote a (cooperative) 
{MA}-C-POMDP 
with $N$ 
agents, state space $\sspace$, joint-observation space $\ospace$, joint-action space $\aspace$, transition-law $\mcl{P}_{tr}$, immediate-cost functions $c$ and $d$, (fixed) initial distribution $P_1$, space of decentralized policy-profiles $\mcl{U}$, and discount factor $\alpha\in\l(0, 1\r)$. The decision problem (to be detailed later on) has the following attributes and notation.
\begin{itemize}
[leftmargin=0pt, 
itemindent=10pt,
labelwidth=0pt, 
labelsep=5pt, 
]
\item \textbf{State Process}: 
The state-space $\sspace$ is some topological space with a Borel $\sigma$-algebra $\borel{\sspace}$. The state-process is denoted by $\l\{\Stt{t} \r\}_{t=1}^{\infty}$.
\item \textbf{Joint-Observation Process}: The joint-observation space $\ospace$ is a countable discrete space of the form $ \ospace = \prod_{n=0}^{N} \onspace{n}$, where $\onspace{0}$ denotes the common observation space of all agents and $\onspace{n}$ denotes the private observation space of agent $n \in [N]$. The joint-observation process is denoted by $\l\{ \Ot{t} \r\}_{t=1}^{\infty}$ where $\Ot{t} = \Otn{t}{0:N}$ and is such that at time $t$, agent $n\in [N]$ observes $\Otn{t}{0}$ and $\Otn{t}{n}$ only.

\item \textbf{Joint-Action Process}: The joint-action space $\aspace $ is a finite discrete space of the form $ \aspace  = \prod_{n=1}^{N} \anspace{n} $, where $\anspace{n}$ denotes the action space of agent $n\in [N]$. The joint-action process is denoted by $\l\{ \At{t} \r\}_{t=1}^{\infty}$ where $\At{t} = \Atn{t}{1:N}$ and $\Atn{t}{n}$ denotes the action of agent $n$ at time $t$.\footnote{The results in this work also hold if for every $( \hstn{h}{t}{0}, \hstn{h}{t}{n}) \in \hstnspace{t}{0}\times\hstnspace{t}{n}$, agent $n$ is allowed to take action from a separate finite discrete space $\anspace{n}(\hstn{h}{t}{0}, \hstn{h}{t}{n})$.} Since all $\anspace{n}$'s and $\aspace$ are finite, they are all compact metric spaces.\footnote{Hence, also complete and separable.} 

\item \textbf{Transition-law}: At time $t \in \mbb{N}$, given the current state $\Stt{t}$ and current joint-action $\At{t}$, the next state $\Stt{t+1}$ and the next joint-observation $\Ot{t+1}$ are determined in a time-homogeneous manner, independent of all previous states, all previous and current joint-observations, and all previous joint-actions. 
The transition-law is given by 
\begin{align*}
   \hspace{-0.10in} \mcl{P}_{tr} \defeq \l\{ P_{saBo}: s \in \sspace, a\in\aspace, B \in  \borel{\sspace}, o \in \ospace \r\},\numberthis\label{eq:transitionlaw}
\end{align*}
where for all $t \in \mbb{N}$, 
\begin{align}
&\pr\l( \Stt{t+1} \in B , \Ot{t+1} = o | \Stt{1:t-1}  = \stt{1:t-1}, \r.\notag\\
&\hspace{12.5pt} \l. \Ot{1:t} = \ot{1:t}, \At{1:t-1} = \at{1:t-1}, 
\Stt{t} = s, \At{t} = a \r)\notag\\
&\hspace{10pt}
= \pr\l( \Stt{t+1} \in B, \Ot{t+1} = o | \Stt{t} = s, \At{t} = a \r) \label{eq:psabo}\\
&\hspace{10pt} 
\defeq P_{saBo}. \notag
\end{align}

\item \textbf{Immediate-costs}: 
The immediate cost $c : \sspace \times \aspace \mapsto \mbb{R}$ is a real-valued function whose expected discounted aggregate (to be defined later) we would like to minimize. On the other hand, the immediate cost $d : \sspace \times \aspace \mapsto \mbb{R}^K$ is $\mbb{R}^K$-valued function whose expected discounted aggregate we would like to keep within a specified threshold. For these reasons, we call $c$ and $d$ as the immediate objective and constraint costs respectively. We shall make use of the following assumption on immediate-costs in Theorem \ref{thm:strongduality}.

\begin{assumption}\label{assmp:boundedcosts}
The immediate objective cost is bounded from below and the immediate constraint costs are bounded, i.e., there exist $\udl{c} \in \mbb{R}$ and $\udl{d}, \ov{d} \in \mbb{R}^K$ such that
\begin{align*}
\udl{c} \le c(\cdot, \cdot) \text{ and } 
\udl{d} \le d(\cdot, \cdot) \le \ov{d}. \numberthis\label{eq:immediatecostsbounded}
\end{align*}
Let $ \ulbar{d} = \| \udl{d} \|_{\infty} \vee  \| \ov{d} \|_{\infty}$ so that $\| d(\cdot, \cdot) \|_{\infty} \le \ulbar{d} < \infty$.
\end{assumption}

\item \textbf{Initial Distribution}: $P_1$ is a (fixed) probability measure for the initial state and initial joint-observation, i.e., $P_1 \in \m{\sspace\times \ospace}$ and
\begin{align}
\begin{split}\label{eq:initialdistribution}
P_1\l( B, o \r) &\defeq \pr\l( \Stt{1} \in B, \Ot{1} = o \r).
\end{split}
\end{align}   

\item \textbf{Space of Policy-Profiles}: At time $t \in \mbb{N}$, the \emph{common history} of all agents is defined as all the common observations received thus far, i.e.,  $\Hstn{t}{0} \defeq \l( \Otn{1:t}{0}  \r)$. Similarly, the \emph{private history} of agent $n \in [N]$ at time $t$ is defined as all observations received and all the actions taken by the agent thus far (except for those that are part of the common information), i.e., 
\begin{align}
    \begin{split}\label{eq:Hstn}
        \Hstn{1}{n} &\defeq \Otn{1}{n} \setminus \Otn{1}{0}, \text{ and}\\
        \Hstn{t}{n} &\defeq \l( \Hstn{t-1}{n}, (\Atn{t-1}{n}, \Otn{t}{n})\setminus \Otn{t}{0} \r) \ \forall t \in [2,\infty]_{\mbb{Z}}.
    \end{split}
\end{align}
Finally, the \emph{joint history} at time $t$ is defined as the tuple of the common history and all the private histories at time $t$, i.e., $\Hst{t}  \defeq \Hstn{t}{0:n}$. 

\hspace{5pt} With the above setup, we define a (decentralized) behavioral policy-profile $u $ as a tuple $\un{u}{1:N} \in \uspace \defeq \prod_{n=1}^{N} \uspace^{(n)} $ where $\un{u}{n}$ denotes some behavioral policy used by agent $n$, i.e., $\un{u}{n}$ itself is a tuple of the form $ \utn{u}{1:\infty}{n}$ where $\utn{u}{t}{n}$ maps $\hstnspace{t}{0} \times \hstnspace{t}{n}$ to $\m{\anspace{n}}$, and where agent $n$ uses the distribution $\utn{u}{t}{n} ( \Hstn{t}{0}, \Hstn{t}{n} )$ to choose its action $\Atn{t}{n}$. We pause to emphasize that in a (decentralized) behavioral policy, at any time $t$, each agent randomizes over its action-set independently of all other agents (\textit{no common randomness is used}). Thus, given a joint-history $\hst{h}{t} \in \hstspace{t}$ at time $t$, the probability that joint-action $\at{t} \in \aspace$ is taken is given by
\begin{align*}
\ut{u}{t}\l(\at{t}|\hst{h}{t}\r) 
&\defeq \prod_{n=1}^{N} \utn{u}{t}{n}\l( \hstn{h}{t}{0}, \hstn{h}{t}{n}  \r) \l( \atn{t}{n} \r)\\
&=\prod_{n=1}^{N} \utn{u}{t}{n}\l( \atn{t}{n} \big| \hstn{h}{t}{0}, \hstn{h}{t}{n}  \r).\numberthis\label{eq:uah}
\end{align*}
\begin{rem}
With Assumption 1, the conditional expectations $\mbb{E}_{P_1} \l[ \cCost \mid \Hst{t} = \hst{h}{t}, \At{t} = \at{t} \r] $ and $\mbb{E}_{P_1} \l[ \dCost \mid \Hst{t} = \hst{h}{t}, \At{t} = \at{t} \r] $ exist, are unique, and are bounded from below. Furthermore, the latter are element-wise finite.
\end{rem}

\item \textbf{Decision Process}: Let $\prup{u}{P_1}$ be the probability measure corresponding to policy-profile $u\in\uspace$ and initial-distribution $P_1$, and let $\E{u}{P_1}$ denote the corresponding expectation operator.\footnote{The existence and uniqueness of $\prup{u}{P_1}$ can be ensured by an adaptation of the Ionesca-Tulcea theorem \cite{tulcea49}.} We define \textit{infinite-horizon expected total discounted costs} $C:\uspace\ra \mbb{R} \cup \{ \infty \}$ and $D:\uspace \ra \mbb{R}^K$ as
\begin{align*}
\fullccosts{u} = \fullccost{u} &\defeq  \E{u}{P_1} \l[ \sum_{t=1}^{\infty} \alpha^{t-1} \cCost \r],
\numberthis\label{eq:C}\\
\text{and}\ 
\fulldcosts{u} = \fulldcost{u} &\defeq  \E{u}{P_1}\l[ \sum_{t=1}^{\infty} \alpha^{t-1} \dCost \r].\numberthis\label{eq:D}
\end{align*}
\begin{rem}\label{rem:real_valued_aggregatecosts}
Assumption \ref{assmp:boundedcosts} ensures that $\fullccosts{u}\in \mbb{R} \cup \{ \infty \}$, and $\fulldcosts{u} \in \mbb{R}^K$ with (absolute) element-wise bound $\ulbar{d}/(1-\alpha)$.
\end{rem}
The decision process proceeds as follows: \textit{i}) At time $t \in \mbb{N}$, the current state $\Stt{t}$ and observations $\Ot{t}$ are generated; \textit{ii}) Each agent $n \in [N]$ chooses an action $\an{n} \in \anspace{n}$ based on $\Hstn{t}{0}, \Hstn{t}{n}$; \textit{iii}) the immediate-costs $c\l( \Stt{t}, \At{t} \r), d\l(\Stt{t}, \At{t}\r)$ are incurred; 
\textit{iv}) The system moves to the next state and observations according to the transition-law $\mcl{P}_{tr}$. 
\end{itemize}
\section{Optimization Problem}\label{sec:optimization_problem}
To formulate the MA-C-POMDP optimization problem, we first need to give a suitable topology to the space of behavioral policy-profiles, in particular, one in which it is compact and convex.\footnote{Convexity is a set property rather than a topological property. In the rest of the paper, by a ``convex topological space'', we mean convexity of the set on which the topology is defined.} To this end, 
we use the finiteness of the action space $\anspace{n} $ and the countability of the joint-observation space $\ospace$ to associate $\uspace$ with a product of compact sets that are parameterized by (countable number of) all possible histories. Tychonoff's theorem 
(see Proposition \ref{prop:tychonoff} )
then helps achieve compactness under the product topology. (Convexity comes trivially). Now, we make this idea precise. For $t\in \mbb{N}$ and $n\in[0,N]_{\mbb{Z}}$, let $\hstnspace{t}{n}$ denote the set of all possible realizations of $\Hstn{t}{n}$. Then, by countability of observation and action spaces, the sets
\begin{align}
\begin{split}\label{eq:hthnandh}
\hstspace{t} &\defeq \prod_{n=0}^{N} \hstnspace{t}{n},\\
\hsnspace{n} &\defeq \bigcup_{t= 1}^{\infty} \hstnspace{t}{0} \times \hstnspace{t}{n}, \text{ and }\\
\hsspace &\defeq \bigcup_{t=1}^{\infty} \hstspace{t},
\end{split}
\end{align}
are countable. Here, $\hstspace{t}$ is the set of all possible joint-histories at time $t$, $\hsnspace{n}$ is the set of all possible histories of agent $n$, and $\hsspace$ is the set of all possible joint-histories. With this in mind, one observes that $\uspace$ is in one-to-one correspondence with the set $\xuspace \defeq \prod_{n=1}^{N} \xuspacen{n}$, where
\begin{align*}
\xuspacen{n} \defeq \prod_{h \in \hsnspace{n}} \m{\anspace{n}; h},\numberthis\label{eq:xuspace}
\end{align*}
and $\m{\anspace{n}; h}$\footnote{$M_1(\cdot)$ denotes the set of all probability measures on $\cdot$.} is a copy of $\m{\anspace{n}}$ dedicated for agent-$n$'s history $h$. For example, a given policy $u$ would correspond to a point $x\in \xuspace$ such that $x_{n, \l( \hstn{h}{t}{0}, \hstn{h}{t}{n}\r) } = \utn{u}{t}{n} \l( \cdot | \hstn{h}{t}{0}, \hstn{h}{t}{n} \r) $, and similarly, vice versa. 

\hspace{5pt} Since $\anspace{n}$ is a complete separable (compact) metric space, by Prokhorov's Theorem 
(see Proposition \ref{prop:prokhorov}),
each $\m{\anspace{n}; h}$ is a compact (and convex\footnote{Convexity of $\m{\anspace{n}}$ is trivial.}) metric space (with the topology of weak-convergence). Therefore, endowing $\xuspacen{n}$ and $\xuspace$ with the product topology makes each a compact (and convex) metric space via Tychonoff's theorem (see 
Proposition \ref{prop:tychonoff}), 
which is also metrizable 
via Proposition \ref{prop:metrizability}. Given the one-to-one correspondence, \textbf{from now onward, we assume that $\uspacen{n}$ and $\uspace$ have the same topology as that of $\xuspacen{n}$ and $\xuspace$ respectively}. Henceforth, we will consider $C$ and $D_k$'s as functions on topological spaces. Furthermore, since $\uspace$ has been shown to be a compact metric space (hence, also complete and separable), we can also define $\borel{\uspace} = \otimes_{n=1}^{N} \borel{\uspacen{n}}$\footnote{For separable metric spaces $\mcl{W}_1, \mcl{W}_2, \ldots$, $\borel{\mcl{W}_1 \times \mcl{W}_2 \times \ldots } = \borel{\mcl{W}_1} \otimes \borel{\mcl{W}_2} \otimes \ldots$. See \cite{kallenberg2002foundations}[Lemma 1.2].}, and $\m{\uspace}$, where $\m{\uspace}$ is compact (and convex) metrizable space by Prokhorov's theorem 
(see Proposition \ref{prop:prokhorov}).

It will be helpful to work with mixtures of behavioral policy-profiles -- wherein the team first uses a measure $\mu \in M_1(\uspace) $ to choose its policy-profile $\policy \in \uspace$ and then proceeds with it from time 1 onward. Under this setup, the policy-profile chosen collectively by the agents becomes a random object, and we extend the definitions of $C$ and $D$ to $\wh{C} : \m{\uspace} \ra \mbb{R} \cup \{\infty\}$ and $\wh{D}: \m{\uspace} \ra \mbb{R}^K$ as follows:
\begin{align}
\begin{split}\label{eq:lagrangianmix}
    \fullccostsmix{\mu} &= \fullccostmix{\mu} \defeq \E{U\sim \mu}{} \l[ C(U) \r], \text{and} \\
    \wh{D} (\mu) &= \fulldcostmix{\mu} \defeq \E{U\sim \mu}{} \l[ D(U) \r].
    \end{split}
\end{align}
The goal of the agents is to work cooperatively to solve the following constrained optimization problem.
\begin{align}
\begin{split}
 &\text{minimize } \fullccostsmix{\mu} \notag\\
 &\text{subject to } \mu \in \m{\uspace} \text{ and } \fulldcostsmix{\mu} \le \constraintv.
\end{split}
 \Bigg\}\tag{\textit{MA-C-POMDP}} \label{eq:macpomdp}
\end{align}
Here, $\constraintv$ is a fixed $K$-dimensional real-valued vector. We refer to the solution of \eqref{eq:macpomdp} as its \emph{optimal value} and denote it by $\optcosts = \optcost$. In particular, if the set of feasible mixtures is empty, we set $\optcosts$ to $\infty$, and, with slight abuse of terminology, we will consider any mixture in $\m{\uspace}$ to be optimal.

\hspace{5pt} The following assumption about feasibility of \eqref{eq:macpomdp} will be used in one of the parts of Theorem \ref{thm:strongduality}.
\begin{assumption}[Slater's Condition]\label{assmp:slatercondition}
There exists a mixture $\mu \in \m{\uspace}$ and $\zeta > 0$ for which
\begin{align*}
    \fulldcosts{\mu} \le \constraintv - \zeta1.\numberthis\label{eq:slatercondition} 
\end{align*}
\end{assumption}
\section{Characterization of Strong Duality}\label{sec:strongduality}
To solve \eqref{eq:macpomdp}, we define the Lagrangian function $\wh{L} : \m{\uspace} \times \mcl{Y} \ra \mbb{R} \cup \{\infty\} $ as follows.
\begin{align*}
\lagsmix{\mu}{\lambda} &= \lagmix{\mu}{\lambda} 
\defeq \wh{C}(\mu) + \dotp{\lambda}{\wh{D}(\mu) -\constraintv}\\
&= \E{U\sim \mu}{} \underbrace{\l[ C(U) + \dotp{\lambda}{D(U) - \constraintv}  \r]}_{\defeq \lag{U}{\lambda}=\lags{U}{\lambda}}.
\numberthis\label{eq:lagrangian}
\end{align*}
Here, $\mcl{Y} \defeq \{ \lambda \in \mbb{R}^K : \lambda \ge 0\}$ is the set of tuples of $K$ non-negative real-numbers, each commonly known as a Lagrange-multiplier. 
Our main result shows that the the solution $\udl{C}$ satisfies
\begin{align*}
\optcosts &= \infsup{\mu\in \m{\uspace}}{\lambda\in \mcl{Y}} \lagsmix{\mu}{\lambda}
,\numberthis\label{eq:optccost:infsup}
\end{align*}
and that the inf and sup can be interchanged, i.e.,
\begin{align*}
\optcosts &= \supinf{\lambda\in \mcl{Y}}{\mu\in \m{\uspace}} \lagsmix{\mu}{\lambda}
.\numberthis\label{eq:optccost:supinf}
\end{align*}

\begin{thm}[Strong Duality and Existence of Saddle Point]\label{thm:strongduality}
Under Assumption \ref{assmp:boundedcosts}, the following statements hold.
\begin{enumerate}[leftmargin=0pt, itemindent=20pt, labelwidth=15pt, labelsep=5pt, listparindent=0.7cm, align=left]
\item[(a)] The optimal value satisfies 
\begin{align*}
\optcosts = \infsup{\mu\in\m{\uspace}}{\lambda\in \mcl{Y}} \lagsmix{\mu}{\lambda}
.\numberthis   
\end{align*}
\item[(b)] A mixture $\mu^\star \in \m{\uspace}$ is optimal if and only if $\optcosts = \sup_{\lambda\in \mcl{Y}} \lagsmix{\mu^\star}{\lambda}$.
\item[(c)] Strong duality holds for \eqref{eq:macpomdp}, i.e.,
\begin{align*}
\optcosts &= \infsup{\mu\in\m{\uspace}}{\lambda\in \mcl{Y}} \lagsmix{\mu}{\lambda}
= \supinf{\lambda\in \mcl{Y}}{\mu\in\m{\uspace}} \lagsmix{\mu}{\lambda}
.\numberthis
\end{align*}
Moreover, there exists a $\mu^\star \in \m{\uspace}$ such that $\optcosts = \sup_{\lambda \in \mcl{Y}} \lagsmix{\mu^\star}{\lambda} $ and $\mu^\star$ is optimal for \eqref{eq:macpomdp}. 
\item[(d)] If Assumption \ref{assmp:slatercondition} holds, then there also exists $\lambda^\star \in \mcl{Y}$ such that the following saddle-point condition holds for all $(\mu,\lambda)\in \m{\uspace} \times \mcl{Y}$,
\begin{align*}
\lagsmix{\mu^\star}{\lambda} \le \lagsmix{\mu^\star}{\lambda^\star} = \optcosts \le \lagsmix{\mu}{\lambda^\star}.
\numberthis\label{eq:saddlepointconditions} 
\end{align*}
i.e., $\mu^\star$ minimizes $\lagsmix{\cdot}{\lambda^\star}$ and $\lambda^\star$ maximizes $\lagsmix{\mu^\star}{\cdot}$. In addition to this, the primal dual pair $\l( \mu^\star, \lambda^\star \r)$ satisfies the complementary-slackness condition:
\begin{align*}\label{eq:compslack}
    \dotp{\lambda^\star}{\fulldcostsmix{\mu^\star}-\constraintv} = 0.\numberthis
\end{align*}
\end{enumerate}
\end{thm}

\begin{proof}
\begin{enumerate}[leftmargin=0pt, itemindent=20pt,
labelwidth=15pt, labelsep=5pt, 
align=left]
\item[(a)] If $\mu \in \m{\uspace}$ is feasible (i.e., it satisfies $\fulldcostsmix{\mu} \le \constraintv$), then the $\sup$ is obtained by choosing $\lambda = 0$, so
\begin{align*}
\sup_{\lambda\in\mcl{Y}} \lagsmix{\mu}{\lambda} &= \fullccostsmix{\mu}.
\numberthis\label{eq:ufeasible}
\end{align*}
If $\mu \in \m{\uspace}$ is not feasible, then 
\begin{align*}
\sup_{\lambda\in \mcl{Y}} \lagsmix{\mu}{\lambda} = \infty.
\numberthis\label{eq:unotfeasible}
\end{align*}
Indeed, suppose WLOG that the $k^{th}$ constraint is violated, i.e., $\fulldkcostsmix{k}{\mu} > \constraintv_k$, then $\infty$ can be obtained by choosing $\lambda_k$ arbitrarily large and setting other $\lambda_k$'s to 0.

\hspace{5pt} From \eqref{eq:ufeasible}, \eqref{eq:unotfeasible}, and our convention that $\optcosts = \infty$ whenever the feasible-set is empty, it follows that
\begin{align*}
\optcosts = \infsup{\mu\in\m{\uspace}}{\lambda\in\mcl{Y} } \lagsmix{\mu}{\lambda}.
\numberthis\label{eq:fullccostisinfsup}
\end{align*}

\item[(b)] By our convention on the value of $\optcosts$ (when there is no feasible mixture), $\mu^\star$ is optimal if and only if $\fullccostsmix{\mu^\star} = \optcosts $, i.e., $\sup_{\lambda\in \mcl{Y} } \lagsmix{\mu^\star}{\lambda} = \optcosts$.

\item[(c)] To establish strong duality, we use 
Proposition \ref{prop:sionminimax} 
which requires $\m{\uspace}$ and $\mcl{Y}$ to be convex topological spaces, with $\m{\uspace}$ being compact as well. It is clear that $\mcl{Y}$ is convex and we can endow it with the usual subspace topology of $\mbb{R}^K$. Convexity of $\m{\uspace}$ is trivial and its compactness has been ensured in Section \ref{sec:optimization_problem}. By definition, $\wh{L}$ is affine and thus trivially concave in $\lambda$. Proposition \ref{prop:integral_linearity} 
implies that $\wh{L}$ is convex in $\mu$ and Lemma \ref{lem:lsc2} shows that $\wh{L}$ is lower semi-continuous\footnote{For definition of lower semi-continuity, see 
Definition \ref{dfn:lsc}.} 
in $\mu$. 
From Proposition \ref{prop:sionminimax}, 
it then follows that
\begin{align*}
\infsup{\mu\in\m{\uspace} }{\lambda\in \mcl{Y}} \lagsmix{\mu}{\lambda} = \supinf{\lambda\in \mcl{Y}}{\mu\in \m{\uspace}} \lagsmix{\mu}{\lambda},
\end{align*}
and that there exists $\mu^\star \in \m{\uspace}$ such that
\begin{align*}
\sup_{\lambda\in \mcl{Y}} \lagsmix{\mu^\star}{\lambda} = \infsup{\mu\in\m{\uspace}}{\lambda\in \mcl{Y}} \lagsmix{\mu}{\lambda}.
\end{align*}
The optimality of $\mu^\star$ is implied by parts (b) and (a). 

\item[(d)] This follows from Lagrange-multiplier theory.
\end{enumerate}
This concludes the proof.
\end{proof}

\begin{lem}[Lower Semi-Continuity of $\wh{L}$ on $\m{\uspace}$]\label{lem:lsc2}
    Under Assumption \ref{assmp:boundedcosts}, $\wh{L}$ is lower semi-continuous on $\m{\uspace}$.
\end{lem}
\begin{proof}
    Fix $\lambda\in\mcl{Y}$ and $\mu\in \m{\uspace}$. Let $\l\{ \mu_i \r\}_{i=1}^{\infty}$ be a sequence of measures in $\m{\uspace}$ that converges to $\mu$. 
    We want to show
    \begin{align*}
        \liminf_{i\ra\infty} \E{U \sim \mu_i}{} \l[  \lags{U}{\lambda} \r] \ge \E{U \sim \mu}{} \l[  \lags{U}{\lambda} \r].
    \end{align*}
    By Lemma \ref{lem:lsc}, $L$ is point-wise lower semi-continuous on $\uspace$. Therefore, 
    Proposition \ref{prop:lsc} 
    applies on $\m{\uspace}$ and the above inequality follows.
\end{proof}

\begin{lem}[Lower Semi-Continuity of $L$ on $\uspace$]\label{lem:lsc}
Under Assumption \ref{assmp:boundedcosts}, the functions $C$ and $D_k$'s are lower semi-continuous on $\uspace$. Hence, $L$ is lower semi-continuous on $\uspace$. 
\end{lem}
\begin{proof}
We will prove the statement for $C$. The proof of lower semi-continuity of $D_k$'s is similar. For brevity, let 
\begin{align*}
\pruphsts{u}{t}{\hst{h}{t}, \at{t}} &= \pruphst{u}{t}{\hst{h}{t}, \at{t}} \defeq \prup{u}{P_1}\l(\Hst{t} = \hst{h}{t}, \At{t} = \at{t}\r),
\\
\zuphsts{u}{t}{\hst{h}{t}, \at{t}} &= \zuphst{u}{t}{\hst{h}{t}, \at{t}}\\
&\hspace{-25pt} \defeq \pruphsts{u}{t}{\hst{h}{t}, \at{t}} \mbb{E}_{P_1}\l[ \cCost | \Hst{t} = \hst{h}{t}, \At{t} = \at{t} \r],
\end{align*}
where we use the convention $0\cdot\infty=0$. Then,
\begin{align*}
\fullccosts{u} &= \E{u}{P_1}\l[ \sum_{t=1}^{\infty} \alpha^{t-1} c(S_t, A_t) \r] \\
&= \E{u}{P_1}\l[ \sum_{t=1}^{\infty} \alpha^{t-1} \l( c(S_t, A_t) - \udl{c} \r) \r] + \sum_{t=1}^{\infty} \alpha^{t-1} \udl{c}\\
&\labelrel{=}{eqr:cp1u:a} \sum_{t=1}^{\infty} \alpha^{t-1} \E{u}{P_1} \l[ c(S_t, A_t) - \udl{c} \r] + \sum_{t=1}^{\infty} \alpha^{t-1} \udl{c}\\
&\labelrel{=}{eqr:cp1u:b} \sum_{t=1}^{\infty} \alpha^{t-1} \E{u}{P_1} \l[ \mbb{E}_{P_1} \l[ c(S_t, A_t)  | \Hst{t}, \At{t} \r] \r]\\
&= \sum_{t=1}^{\infty} \sum_{\hst{h}{t} \in \hstspace{t}} \sum_{\at{t}\in \aspace} \alpha^{t-1} \zuphsts{u}{t}{\hst{h}{t}, \at{t}}.
\end{align*}
Here, \eqref{eqr:cp1u:a} follows from applying the Monotone-Convergence Theorem to the (increasing non-negative) sequence $\{ \sum_{t=1}^{i} \alpha^{t-1} \l( c\l( \Stt{t}, \At{t} \r) - \udl{c} \r) \}_{i=1}^{\infty}$ 
(see Proposition \ref{prop:mct}); 
and \eqref{eqr:cp1u:b} uses the tower property of conditional expectation.\footnote{The conditional expectations $\mbb{E}_{P_1} \l[ c(S_t, A_t) | \Hst{t}, \At{t} \r]$ exist and are unique because $c(\cdot, \cdot)$ is bounded from below.}

Let $\l\{ \useq{i}{u} \r\}_{i=1}^{\infty}$ be a sequence in $\uspace$ that converges to $u$. By Fatou's Lemma 
(see Proposition \ref{prop:fatou}), 
\begin{align*}
\liminf_{i\ra \infty} \fullccosts{\useq{i}{u}} \ge \sum_{t=1}^{\infty} \sum_{\hst{h}{t} \in \hstspace{t}} \sum_{\at{t}\in \aspace} \alpha^{t-1} \liminf_{i\ra\infty} \zuphsts{\useq{i}{u}}{t}{\hst{h}{t}, \at{t}}.\numberthis\label{eq:step1}
\end{align*}
Following Lemma \ref{lem:puth}, 
$\pruphsts{\useq{i}{u}}{t}{\hst{h}{t}, \at{t}}\ge 0 $ converges to $\pruphsts{u}{t}{\hst{h}{t}, \at{t}}$. Therefore,
\begin{align*}
\liminf_{i\ra\infty} \zuphsts{\useq{i}{u}}{t}{\hst{h}{t}, \at{t}} \ge \zuphsts{u}{t}{\hst{h}{t}, \at{t}}.\numberthis\label{eq:step2}
\end{align*}
Then, \eqref{eq:step1} and \eqref{eq:step2} result in
$\liminf_{i\ra \infty} \fullccosts{\useq{i}{u}} \ge \fullccosts{u},$ 
which establishes the lower semi-continuity of $\fullccosts{u}$. 
\end{proof}

\begin{lem}\label{lem:puth}[Limit Probabilities for a converging sequence of policy-profiles]
Let $\l\{ \useq{i}{u} \r\}_{i=1}^{\infty}$ be a sequence in $\uspace$ that converges to $u$. Then, for any $t \in \mbb{N}$, $ \hst{h}{t} \in \hstspace{t} $, and $\at{t} \in \mcl{A}$,
\begin{align*}
\lim_{i\ra \infty}  \pruphsts{\useq{i}{u}}{t}{\hst{h}{t}, \at{t}} = \pruphsts{u}{t}{\hst{h}{t}, \at{t}},
\end{align*}
where $\pruphsts{\cdot}{t}{\hst{h}{t}, \at{t}} = \prup{\cdot}{P_1} \l( \Hst{t} = \hst{h}{t}, \At{t} = \at{t} \r)$. In other words, for every $t \in \mbb{N}$, the sequence of measures $\l\{ \pruphsts{ \useq{i}{u}}{t}{\cdot, \cdot} \r\}_{i=1}^{\infty}$ converges weakly to $\pruphsts{u}{t}{\cdot, \cdot}$.
\end{lem}
\begin{proof}
Given that $\useq{i}{u}$ converges to $u$, by Proposition \ref{prop:conv_in_prod_topology}, for every $n \in [N]$, $\useq{i}{\utn{u}{t}{n}} (\hstn{h}{t}{0}, \hstn{h}{t}{n} )$ converges weakly to $ \utn{u}{t}{n} ( \hstn{h}{t}{0}, \hstn{h}{t}{n} ) $. Since $\mcl{A}^n$ is finite, this means that for each $\an{n}\in \anspace{n}$, $\useq{i}{\utn{u}{t}{n}} ( \an{n} | \hstn{h}{t}{0}, \hstn{h}{t}{n} )$ converges to $\utn{u}{t}{n} ( \an{n} | \hstn{h}{t}{0}, \hstn{h}{t}{n} )$, which further implies that for all $a \in \aspace$, $ \useq{i}{\ut{u}{t}} ( a | \hst{h}{t}) $ converges to $ \ut{u}{t} ( a | \hst{h}{t}) $. Now, we use forward induction to prove the statement. 
\begin{enumerate}
[leftmargin=0pt, 
itemindent=10pt,
labelwidth=0pt, 
labelsep=5pt, 
]
\item \textbf{Base Case}:  For time $t=1$, let $\ot{1} \in \hstspace{1} = \ospace$ and $\at{1} \in \mcl{A}$. We have
\begin{align*}
\pruphsts{\useq{i}{u}}{1}{\ot{1}, \at{1}}
=P_1\l( \sspace, o \r) \useq{i}{\ut{u}{1}} \l( \at{1} | \ot{1} \r) 
\ra \pruphsts{u}{1}{\ot{1}, \at{1}}.
\end{align*}

\item \textbf{Induction Step}: Assuming that the statement is true for time $t$, we show that it is true for time $t+1$. Let $\hst{h}{t+1} = \l( \ot{1:t+1}, \at{1:t} \r) = \l( \hst{h}{t}, \at{t}, \ot{t+1} \r) \in \hstspace{t+1}$ and $\at{t+1} \in \aspace$. We have
\begin{align*}
&\pruphsts{\useq{i}{u}}{t+1}{\hst{h}{t+1}, \at{t+1}} =  
\pruphsts{\useq{i}{u}}{t}{\hst{h}{t}, \at{t}} \\
&\hspace{25pt} \times \useq{i}{\ut{u}{t+1}} \l( \at{t+1} | \hst{h}{t+1} \r) \pr_{P_1} \l( \ot{t+1} | \hst{h}{t}, \at{t} \r).
\end{align*}
By inductive hypothesis, $\pruphsts{\useq{i}{u}}{t}{\hst{h}{t}, \at{t}} $ converges to $\pruphsts{u}{t}{\hst{h}{t}, \at{t}}$, and $ \useq{i}{\ut{u}{t}}\l( \at{t+1} | \hst{h}{t+1}\r) $ converges to $ \ut{u}{t} \l( \at{t+1} | \hst{h}{t+1}\r) $ by assumption. We conclude that $\pruphsts{\useq{i}{u}}{t+1}{\hst{h}{t+1}, \at{t+1}}$ converges to $\pruphsts{u}{t+1}{\hst{h}{t+1}, \at{t+1}}$.
\end{enumerate}
This completes the proof.
\end{proof}

\section{Conclusion}\label{sec:conclusion}
In this work, we studied a (cooperative) multi-agent constrained POMDP in the setting of infinite-horizon expected total discounted costs. We established strong duality and existence of a saddle point using an extension of Sion's Minimax Theorem which required giving a suitable topology to the space of all possible policy-profiles and then establishing lower semi-continuity of the Lagrangian function. The strong duality result provides a firm theoretical footing for future development of primal-dual type planning and learning algorithms for 
{MA}-C-POMDPs---see~\cite{khan2023cooperative} for one such algorithm. 
\appendices


\renewcommand{\theequation}{\thesection.\arabic{equation}}
\section{Helpful Facts and Results}\label{sec:appendix:helpful_facts}
\begin{dfn}[Semi-continuity]\label{dfn:lsc}
A function $f : \mcl{X} \mapsto [-\infty, \infty]$ on a topological space $\mcl{X}$ is called \emph{lower semi-continuous} if for every point $x_0 \in \mcl{X}$, 
$\liminf\limits_{x\ra x_0} f(x) \ge f(x_0)$. 
We call $f$ \emph{upper semi-continuous} function if $-f$ is lower semi-continuous.
\end{dfn}

\begin{prop}[Monotone Convergence Theorem]\label{prop:mct}
Let $\l(X, \mcl{M}, \mu \r)$ be a measure-space. Let $\l\{ f_i \r\}_{i=1}^{\infty}$ be an increasing sequence of measurable functions which are uniformly bounded-from-below. Then, 
\begin{align*}
&\int_{X} \lim_{i\ra\infty} f_i(x) \mu(dx) = \lim_{i\ra\infty} \int_{X} f_i(x) \mu(dx). 
\end{align*}
\end{prop}

\begin{prop}[Convergence in Product Topology]\label{prop:conv_in_prod_topology}
Let $ \{^ix\}_{i=1}^{\infty}$ be a sequence of points of the product space $\prod_{j} X_j$. Then $\{^ix\}_{i=1}^{\infty}$ converges to a point $x \in \prod_{j} X_j$ if and only if the sequence $\{ \pi_j(^ix) \}_{i=1}^{\infty}$ converges to $\pi_j (x) $ for each $j$.
\end{prop}

\begin{prop}[Fatou's Lemma]\label{prop:fatou}
Let  $(X, \mcl{M}, \mu)$ be a measure-space and let $\{ f_i \}_{i=1}^{\infty}$ be a sequence of measurable functions which are uniformly bounded from below. Then,
\begin{align*}
& \liminf_{i\ra\infty} \int f_i(x) \mu (dx)\ge \int \liminf_{i\ra\infty} f_i(x) \mu(dx).
\end{align*}
\end{prop}

\begin{prop}[Tychonoff's Theorem]\label{prop:tychonoff}
Product of a collection of compact spaces is compact under the product topology.
\end{prop}

\begin{prop}[Metrizability of Product Topology on Countable Product of Metric Spaces]\label{prop:metrizability}
Product of countable number of metric spaces, when endowed with the product topology, is metrizable.
\end{prop}

\begin{prop}[Prokhorov's Theorem]\label{prop:prokhorov}
Let $\l( \mcl{X}, \metric{\mcl{X}} \r)$ be a complete separable metric space with distance metric $\metric{\mcl{X}}$ and let $\borel{\mcl{X}}$ denote the Borel $\sigma$-algebra generated by $\metric{\mcl{X}}$. Let $\m{\mcl{X}}$ be the set of all probability measures on $\borel{\mcl{X}}$ and let $\tau$ denote the topology of weak-convergence on $\m{\mcl{X}}$. Then,  
\begin{enumerate}
\item The topological space $ \l(\m{\mcl{X}} , \tau\r)$ is completely-metrizable, i.e., there exists a complete metric $\metric{\m{\mcl{X}}}$ on $ \m{\mcl{X}}$ that induces the same topology as $ \tau $.
\item An arbitrary collection $W \subseteq \m{\mcl{X}}$ of probability measures in $ \m{\mcl{X}}$ is tight iff its closure in $\tau $ is compact (i.e., $W$ is precompact in $\tau$).
\end{enumerate}
\end{prop}

\begin{prop}[Hyperplane Separation Theorem]\label{prop:separation_theorem}
Let $M$ be a non-empty convex subset of 
$\mbb{R}^n$. If $x_0 \in \mbb{R}^n$ does not belong to $M$, there exists $\rho \in \mbb{R}^n$ such that
\begin{align*}
\rho \neq 0 \text { and } \inf_{x \in M} \dotp{p}{x} \geq \dotp{p}{x_0}.
\end{align*}
\end{prop}

\begin{prop}[Integral of Bounded-from-Below function with respect to Convex Combination of Non-negative Measures]\label{prop:integral_linearity}
Let $\l(X, \mcl{M}\r)$ be a measure-space. Let $f : X \ra \mbb{R} \cup \{ \infty \}$ be a measurable function that is bounded from below, and let $\mu, \nu$ be two non-negative measures on $\mcl{M}$. Then, for any $\theta \in [0,1]$,
\begin{align*}
&\int f(x) \l(\theta \mu + (1-\theta) \nu \r)(dx) \\
&\hspace{10pt} = \theta \int f(x) \mu(dx) + (1-\theta) \int f(x)\nu(dx). 
\end{align*}
\end{prop}

\begin{prop}[Behavior of Integrals of a Bounded-from-Below and Lower Semi-Continuous Function]\label{prop:lsc}
Let $(\mcl{X}, \metric{\mcl{X}})$ be a 
complete separable metric space with distance metric $\metric{\mcl{X}}$ and let $\borel{\mcl{X}}$ denote the Borel $\sigma$-algebra generated by $\metric{\mcl{X}}$. Let $\l( \m{\mcl{X}} , \metric{\m{\mcl{X}}} \r)$ be the complete metric space of all probability measures on $\borel{\mcl{X}}$ with the topology of weak-convergence.\footnote{Prokhorov's theorem (see Proposition \ref{prop:prokhorov}) ensures completeness and metrizability of $\m{\mcl{X}}$.} Let $\mu \in \m{\mcl{X}}$ and let $f : \mcl{X} \ra \mbb{R} \cup \l\{ \infty\r\} $ be a function that is lower semi-continuous $\mu$-amost-everywhere\footnote{Lower semi-continuity of $f$ ensures that it is measurable.} and is bounded from below. Then, the function
\begin{align*}
H : \m{\mcl{X}} \mapsto \mbb{R} \cup \l\{ \infty \r\},\  
H(\mu') \defeq \int f(x) \mu'(dx)
\end{align*}
is lower semi-continuous at $\mu$. In particular, if $f$ is point-wise lower semi-continuous, then $H$ is also point-wise lower semi-continuous (on $\m{\mcl{X}}$).
\end{prop}
\begin{proof}
Define $f' : \mcl{X} \ra \mbb{R} \cup \{\infty \}$ as $f'(x) \defeq f(x) \wedge \liminf_{y\ra x} f(y)$. Then, $f'$ minorizes $f$\footnote{That is, $f'(x) \le f(x)$.}, is lower semi-continuous, and coincides with $f$ at $x$ if and only if $f$ is lower semi-continuous at $x$. Also, $f'$ is bounded from below (since $f$ is). By Proposition \ref{prop:lsc3}, $f'$ can be written as the point-wise limit of increasing sequence of uniformly bounded-from-below continuous functions from $\mcl{X}$ into $\mbb{R} \cup \{ \infty \}$, say $\l\{ g_i 
\r\}_{i=1}^{\infty} $, i.e., $f'(x) = \lim_{i\ra \infty} g_i(x)$. Then, for every $\mu' \in \m{\mcl{X}}$,
\begin{align*}
\int f'(x)\mu'(dx) = \int \lim_{i\ra \infty} g_i(x) \mu'(dx) = \lim_{i\ra \infty} \int g_i(x)\mu'(dx),
\end{align*}
where the last equality follows from the Monotone Convergence Theorem (see Proposition \ref{prop:mct}). The above equality shows that the function $H' : \m{\mcl{X}} \ra \mbb{R} \cup \{ \infty\}$ such that $H'(\mu') = \int f'(x) \mu'(dx)$, is the point-wise limit of an increasing sequence of uniformly bounded-from-below continuous functions. Therefore, by Proposition \ref{prop:lsc3}, $H'$ is lower semi-continuous. Now, if $f$ is lower semi-continuous $\mu$-almost-everywhere, then $f = f'$ $\mu-$almost-everywhere. This gives,
\begin{align*}
H(\mu) &= \int f(x) \mu(dx) \\
&= \int f'(x) \mu(dx) \\
&\labelrel{=}{eqr:lsc:H2islsc} \liminf_{\mu'\ra\mu} H'(\mu') \\
&\labelrel{\le}{eqr:lsc:H2minorizesH} \liminf_{\mu'\ra\mu} H(\mu'),
\end{align*}
Here, \eqref{eqr:lsc:H2islsc} uses lower semi-continuity of $H'$ and \eqref{eqr:lsc:H2minorizesH} follows from the fact that $H'$ minorizes $H$ (since $f'$ minorizes $f$). The inequality $H(\mu) \le \liminf_{\mu'\ra\mu} H(\mu')$ is the definition of lower semi-continuity at $\mu$. 
\end{proof}

\begin{prop}[Equivalent Characterization of a Bounded-from-Below Lower Semi-Continuous Function]\label{prop:lsc3}
Let $\l( \mcl{X}, \metric{\mcl{X}} \r)$ be a metric space. Then, a function $f : \mcl{X} \ra \mbb{R} \cup \{ \infty \}$ is a bounded-from-below lower semi-continuous function if and only if it can be written as the point-wise limit of an increasing sequence of uniformly bounded-from-below continuous functions from $\mcl{X}$ into $\mbb{R} \cup \{ \infty \}$. 
\end{prop}
\begin{proof}
\textbf{Necessity}: Define $f_n : \mcl{X} \ra \mbb{R} \cup \{ \infty \}$ as follows:
\begin{align*}
f_n\l( x \r) &\defeq \inf_{y\in\mcl{X}} \l\{ f(y) + n \metric{\mcl{X}} \l(x, y\r) \r\}.
\end{align*}
\begin{enumerate}
\item \textit{Increasing}: 
\begin{align*}
f_{n+1}\l( x \r) = \inf_{y\in\mcl{X}} \l\{ f(y) + (n+1)\metric{\mcl{X}}\l( x,y\r) \r\} \ge f_n(x).
\end{align*}
\item \textit{Uniformly Bounded-from-Below}: Since $f_n\l(x \r) \ge \inf_{y\in\mcl{X}} \l\{ f(y) \r\}$ and $f$ is bounded-from-below, the functions $\l\{ f_n\r\}_{n=1}^{\infty}$ are uniformly bounded-from-below.
\item \textit{Continuity}: By triangle-inequality,
\begin{align*}
f(y) + n\metric{\mcl{X}}\l(y, z\r) \le
f(y) + n\metric{\mcl{X}}\l(y, w\r) +  n\metric{\mcl{X}}\l(w, z\r),
\end{align*}
and therefore, taking the infimum over $y$ on both sides gives $ f_n\l( z \r) - f_n\l( w \r) \le n\metric{\mcl{X}}\l( w, z\r) $. Similarly, we can get $ f_n\l( w \r) - f_n\l( z \r) \le n\metric{\mcl{X}}\l( w, z\r) $, and so
\begin{align*}
|f_n\l( z \r) - f_n\l( w \r)| \le n \metric{\mcl{X}} \l(w, z\r).
\end{align*}
The above relation shows that $f_n$ is Lipschitz and thus continuous.
\item \textit{Point-wise Convergence to $f$}: Fix $x_0 \in \mcl{X}$ and $\eps>0$. We would like to show that there exists a positive integer $n' = n'(x_0, \eps)$ such that, for all $ n \ge n'$, $| f_n\l(x_0\r) - f\l(x_0\r) | < \eps$. Since $f$ is lower semi-continuous at $x_0$, there exists $\delta = \delta(x_0, \eps) > 0$ such that
\begin{align*}
\metric{\mcl{X}}\l( x_0, y\r) < \delta \implies f(y) >  f(x_0) -\eps.\numberthis\label{eq:lsc2:implication}
\end{align*}
Since $f$ is bounded-from-below (and $\delta>0$), there exists a positive integer $n'=n'(\delta(x_0,\eps))$ such that
\begin{align*}
&\metric{\mcl{X}}\l( x_0, y\r) \ge \delta\\
&\hspace{0pt} \implies \forall\  n\ge n', f(y) + n\metric{\mcl{X}}\l( x_0, y\r) > f(x_0)\\
&\hspace{0pt} \implies \forall\  n\ge n', \\
&\hspace{40pt} \inf_{\metric{\mcl{X}}\l( x_0, y\r) \ge \delta } \l\{ f(y) + n\metric{\mcl{X}}(x_0, y) \r\}\ge f\l(x_0\r).
\end{align*}
So, for all $n\ge n'$, we have
\begin{align*}
f(x_0) \ge f_n\l( x_0 \r) &= \inf_{\metric{\mcl{X}}\l( x_0, y\r) \le \delta } \l\{ f(y) + n\metric{\mcl{X}}(x_0, y) \r\}\\
&\ge\inf_{\metric{\mcl{X}}\l( x_0, y\r) \le \delta } \l\{ f(y) \r\}\\
&\labelrel{>}{eqr:lsc2:1}\inf_{\metric{\mcl{X}}\l( x_0, y\r) \le \delta } \l\{ f(x_0) - \eps \r\}\\
&=f(x_0) - \eps.
\end{align*}
where \eqref{eqr:lsc2:1} uses \eqref{eq:lsc2:implication}.
\end{enumerate}
\textbf{Sufficiency}: Let $\l\{ f_n \r\}_{n=1}^{\infty} $ be an increasing sequence of uniformly bounded-from-below continuous functions from $\mcl{X}$ into $\mbb{R} \cup \l\{ \infty \r\}$. Since the sequence is monotonic, it has a point-wise-limit $f : \mcl{X} \ra \mbb{R} \cup \l\{ \infty \r\}$ which is bounded-from-below because all the functions in the sequence are uniformly bounded-from-below. We need to show that $f$ is lower semi-continuous. 

Fix $x_0 \in \mcl{X}$ and $\eps>0$. We would like to show that there exists $\delta = \delta(x_0,\eps)>0$ such that $\metric{\mcl{X}}\l( x_0, y\r) < \delta \implies f(y) >  f(x_0) -\eps $. Since  $\l\{ f_n \r\}_{n=1}^{\infty} $ is increasing (and converges point-wise to $f$), there exists a positive integer $n'=n'(x_0, \eps)$ such that, for all $n\ge n'$, $f(x_0) \ge f_n(x_0) \ge f(x_0) - \frac{\eps}{2}$. Since $f_{n'}$ is lower semi-continuous, there exists $\delta=\delta(n'(x_0, \eps)) > 0$ such that $\metric{\mcl{X}}\l( x_0, y\r)<\delta \implies f(y) \ge f_{n'}(y) > f_{n'}(x_0) - \frac{\eps}{2} \ge f(x_0) - \eps$. 
\end{proof}

\section{A Minimax Theorem for Functions with Positive Infinity
}\label{sec:appendix:minimax}

\begin{prop}[A Minimax Theorem For Functions with Positive Infinity]\label{prop:sionminimax}
Let $\mcl{X}$ and $\mcl{Y}$ be convex topological spaces where $\mcl{X}$ is also compact. Consider a function $f : \mcl{X} \times \mcl{Y} \ra \mbb{R} \cup \{ \infty \} $ such that
\begin{enumerate}
\item for each $y \in \mcl{Y}$, $f\l(\cdot, y \r)$ is convex and lower semi-continuous.
\item for each $x \in \mcl{X}$, $f\l(x, \cdot \r)$ is concave.
\item If $f (x, y) = \infty$, then $f(x, y') = \infty$ for all $y'\in\mcl{Y}$.
\end{enumerate}
Then, there exists $x^\star \in \mcl{X}$ such that
\begin{align*}
\sup_{y\in \mcl{Y}} f\l( x^\star, y \r) &=
\inf_{x \in \mcl{X}} \sup_{y \in \mcl{Y}} f\l( x, y \r)\\
&=\sup_{y \in \mcl{Y}} \inf_{x \in \mcl{X}} f(x, y).
\end{align*}
\end{prop}

Proposition \ref{prop:sionminimax} is a mild adaptation of the Minimax theorem presented in \cite{aubin_book_2002}[Theorem 8.1] where a real-valued function is considered. In the MA-C-POMDP model described in Section \ref{sec:problem}, it is possible that $\fullccosts{u}$ and hence $\lags{u}{\lambda}$ is $\infty$ for all $\lambda \in \mcl{Y}$. We will use the same methodology as in \cite{aubin_book_2002}[Propositions 8.2 and 8.3] to prove Proposition \ref{prop:sionminimax}. In particular, the entire proof remains the same except that in Lemma \ref{lem:lem8.2:aubin:modified}, the compactness of $\mcl{X}$ is used together with Assumption 3). 

Define
\begin{align*}
f^{\sharp}(x) & :=\sup_{y \in \mcl{Y}} f(x, y), & & v^{\sharp}:=\inf_{x \in \mcl{X}} \sup_{y \in \mcl{Y}} f(x, y) \numberthis\\
f^b(y) & :=\inf_{x \in \mcl{X}} f(x, y), & & v^{\flat}:=\sup_{y \in \mcl{Y}} \inf_{x \in \mcl{X}} f(x, y).\numberthis
\end{align*}
To show the equality of $v^{\sharp}$ and $v^{\flat}$, we will introduce an intermediate value $v^{\natural}$ ($v$ natural) and prove successively that $v^{\natural}=v^{\sharp}$ and that $v^{\natural}=v^{\flat}$. 

We denote the family of finite subsets $J$ of $\mcl{Y}$ by $\mcl{J}$. We set
$$
v_J^{\sharp}:=\inf_{x \in \mcl{X}} \sup_{y \in J} f(x, y)
$$
and
$$
v^{\natural}:=\sup_{J \in \mcl{J}} v_J^{\sharp}=\sup_{J \in \mcl{J}} \inf_{x \in \mcl{X}} \sup_{y \in J} f(x, y).
$$
Since every point $y$ of $\mcl{Y}$ may be identified with the finite subset $\{y\} \in \mcl{J}$, we note that $v_{\{y\}}^{\sharp}=f^b(y)$ and consequently, $v^{\flat}=\sup_{y \in \mcl{Y}} v_{\{y\}}^{\sharp} \leq \sup_{J \in \mcl{J}} v_J^{\sharp} = v^{\natural}$. Also, since $\sup_{y \in J} f(x, y) \leq \sup_{y \in \mcl{Y}} f(x, y)$, we deduce that $v_J^{\sharp} \leq v^{\sharp}$, and hence $v^{\natural} \leq v^{\sharp}$. In summary, we have shown that
\begin{align*}
v^{\flat} \leq v^{\natural} \leq v^{\sharp} .
\end{align*}
Lemma \ref{lem:prop8.2:aubin} shows that $v^{\natural} = v^{\sharp} $ and Lemma \ref{lem:prop8.3:aubin} shows that $v^{\flat} = v^{\natural}$. This concludes the proof.

\begin{lem}\label{lem:prop8.2:aubin}
Consider a function $f : \mcl{X} \times \mcl{Y} \mapsto \mbb{R} \cup \{ \infty \} $ such that $\mcl{X}$ is compact and for each $
y \in \mcl{Y}$, $f(\cdot, y)$ is lower semi-continuous. Then, there exists $x^\star \in \mcl{X}$ such that
$$
\sup_{y \in \mcl{Y}} f(x^\star, y)=v^{\sharp}
$$
and
$$
v^{\natural}=v^{\sharp} .
$$
\end{lem}

\begin{rem}
Since the functions $f(\cdot, y)$ are lower semi-continuous, the same is true of the function $f^{\sharp}$.\footnote{Supremum of arbitrary collection of lower semi-continuous functions is lower semi-continuous.} Since $\mcl{X}$ is compact, Weierstrass's theorem implies the existence of $x^\star \in \mcl{X}$ which minimises $f^{\sharp}$. Following (3), this may be written as
\begin{align*}
&\sup_{y \in \mcl{Y}} f(x^\star, y) = f^{\sharp}(x^\star) = \inf_{x \in \mcl{X}} f^{\sharp}(x) \\
&\hspace{50pt} = \inf_{x \in \mcl{X}} \sup_{y \in \mcl{Y}} f(x, y)=v^{\sharp}.
\end{align*}
In comparison to this, Lemma \ref{lem:prop8.2:aubin} proves that $v^{\natural} = v^{\sharp} $.
\end{rem}

\begin{proof}
It suffices to show that there exists $x^\star \in \mcl{X}$ such that
\begin{align*}
\sup_{y \in \mcl{Y}} f(x^\star, y) \leq v^{\natural}.\numberthis\label{eq:vsharp<=vnatural}
\end{align*}
Since $v^{\sharp} \leq \sup_{y \in \mcl{Y}} f(x^\star, y)$ and $v^{\natural} \leq v^{\sharp}$, we shall deduce that $v^{\natural}=v^{\sharp}$.
We set
$$
S_{y}:=\l\{x \in \mcl{X} \mid f(x, y) \leq v^{\natural}\r\}.
$$
The inequality \eqref{eq:vsharp<=vnatural} is equivalent to the inclusion
\begin{align*}
x^\star \in \bigcap_{y \in \mcl{Y}} S_{y}.\numberthis\label{eq:nonemptyintersection}
\end{align*}
Thus, we must show that this intersection is non-empty.
For this, we shall prove that the $S_{y}$ are closed sets (inside the compact set $\mcl{X}$) with the finite-intersection property.\footnote{The intersection of an arbitrary collection of closed sets that lie inside a compact set and satisfy the finite-intersection property, is non-empty.}

If $v^{\natural} = \infty$, then every $S_y$ equals $\mcl{X}$ and the intersection is trivially non-empty. Therefore, WLOG, assume that $v^{\natural}$ is finite. Then the set $S_{y}$ is a lower section of the lower semi-continuous function $f(\cdot, y)$ and is thus closed.\footnote{The lower section of a lower semi-continuous function is closed. For every $\eta \in \mbb{R}$, the corresponding lower section is defined as $\{x \in \mcl{X} : f(x) \le \eta \}$.}

We show that for any finite sequence $J :=\l\{y_{1}, y_{2}, \ldots, y_{n}\r\} \in \mcl{J}$ of $\mcl{Y}$, the finite intersection
$$
\bigcap_{i \in [n]} S_{y_i} \neq \emptyset
$$
is non-empty. In fact, since $\mcl{X}$ is compact, and since $\max_{y \in J} f(\cdot, y) $ is lower semi-continuous, it follows that there exists $\hat{x} \in \mcl{X}$ which minimises this function. Such an $\hat{x} \in \mcl{X}$ satisfies
\begin{align*}
\max_{y \in J} f(\hat{x}, y) &= \inf_{x \in \mcl{X}} \max_{y \in J} f(x, y) \\
&\leq \sup_{J \in \mcl{J}} \inf_{x \in \mcl{X}} \max_{y \in J} f(x, y)=v^{\natural} .
\end{align*}
Since $\mcl{X}$ is compact, the intersection of the closed sets $S_{y}$ is non-empty and there exists $x^\star \in \mcl{X}$ satisfying \eqref{eq:nonemptyintersection} and thus \eqref{eq:vsharp<=vnatural}.
\end{proof}

\begin{lem}\label{lem:prop8.3:aubin}
Consider a function $f : \mcl{X} \times \mcl{Y} \mapsto \mbb{R} \cup \{ \infty \} $ such that $\mcl{X}$ and $\mcl{Y}$ are convex sets, (i) for each $y \in \mcl{Y}$, $f(\cdot, y)$ is convex, and (ii) for each $x \in \mcl{X}$, $f(x, \cdot)$ is concave. Then, $v^{\flat}=v^{\natural}$.
\end{lem}
\begin{proof}
We set $M_J:=\l\{\lambda \in \mbb{R}_{\ge 0}^{|J|} \mid \sum_{i=1}^n \lambda_i=1\r\}$. With any finite (ordered) subset $J \defeq \l\{y_1, y_2, \ldots, y_n\r\}$, we associate the mapping $\phi_J$ from $\mcl{X}$ to $\l( \mbb{R} \cup \{ \infty \} \r)^{|J|}$ defined by
$$
\phi_J(x):=\l(f\l(x, y_1\r), \ldots, f\l(x, y_n\r)\r)
$$
We also set
$$
w_J:=\sup_{\lambda \in M_J} \inf_{x \in \mcl{X}} \dotp{\lambda}{\phi_J(x)}
$$
We prove successively that
\begin{enumerate}
    \item $\sup_{J\in\mcl{J}} w_J \leq v^{\flat}$ (Lemma \ref{lem:lem8.3:aubin}).
    \item $\sup_{J\in\mcl{J}} v^{\sharp}_{J}  \le \sup_{J\in\mcl{J}} w_J$ (Lemma \ref{lem:lem8.2:aubin:modified}).
\end{enumerate}
Hence, the inequalities
\begin{align*}
v^{\natural} = \sup_{J \in \mcl{J}} v_J^{\sharp} \leq \sup_{J \in \mcl{J}} w_J \leq v^{\flat} \leq v^{\natural}
\end{align*}
imply the desired equality 
 $v^{\flat}=v^{\natural}$.
\end{proof}

\begin{lem}\label{lem:lem8.3:aubin}
Consider a function $f : \mcl{X} \times \mcl{Y} \mapsto \mbb{R} \cup \{ \infty \} $ such that $\mcl{Y}$ is convex and for each $x \in \mcl{X}$, $f(x, \cdot)$ is concave. Then, for any finite subset $J$ of $\mcl{Y}$, we have $w_J \leq v^{\flat}$. Hence, $$\sup_{J\in\mcl{J}} w_J \le v^{\flat}.$$    
\end{lem}
\begin{proof}
With each $\lambda \in M_J$, we associate the point $y_\lambda:=\sum_{i=1}^n \lambda_i y_i$ which belongs to $\mcl{Y}$ since $\mcl{Y}$ is convex. The concavity of the functions $\l\{ f(x, \cdot)\r\}_{x\in\mcl{X}}$ implies that
\begin{align*}
\forall x \in \mcl{X}, \quad \sum_{i=1}^n \lambda_i f\l(x, y_i\r) \leq f\l(x, y_\lambda\r).
\end{align*}
Consequently,
\begin{align*}
\inf_{x \in \mcl{X}} \sum_{i=1}^n \lambda_i f\l(x, y_i\r) &\leq \inf_{x \in \mcl{X}} f\l(x, y_\lambda\r) \\
&\leq \sup_{y \in \mcl{Y}} \inf_{x \in \mcl{X}} f(x, y) \defeq v^{\flat}.
\end{align*}
The proof is completed by taking the supremum over $M_J$.
\end{proof}

\begin{lem}\label{lem:lem8.2:aubin:modified}
Consider a function $f : \mcl{X} \times \mcl{Y} \mapsto \mbb{R} \cup \{ \infty \} $ such that $\mcl{X}$ is a convex compact topological space, for each $y \in \mcl{Y}$, $f(\cdot, y)$ is convex and lower semi-continuous, and $f (x, y) = \infty$ implies $f(x, y') = \infty$ for all $y'\in\mcl{Y}$.
Then, 
\begin{align*}
v^{\natural} \defeq \sup_{J \in \mcl{J}} v_J^{\sharp} \leq \sup_{J \in \mcl{J}} w_J .
\end{align*}
\end{lem}
\begin{proof}
WLOG we assume that $\sup_{J \in \mcl{J}} w_J < \infty$. In this case, we can rewrite $w_J$ as $\supinf{\lambda\in M_J}{x\in \mcl{X}_J} \dotp{\lambda}{\phi_J(x)}$ where 
$$\mcl{X}_J \defeq \bigcap_{y\in J} dom f(\cdot, y).$$
To see this, note that $\dotp{\lambda}{\phi_J(x)} $ is a lower semi-continuous function on the compact space $\mcl{X}$. By Weierstrass theorem, $\dotp{\lambda}{\phi_J(x)} $ achieves its minimum in $\mcl{X}$ and we can write $w_J = \sup_{\lambda \in M_J} \dotp{\lambda}{\phi_J(\hat{x}(\lambda))}$. Suppose that $\hat{x}(\lambda) \in \mcl{X} \setminus \mcl{X}_J$, i.e., there exists $y \in J$ such that $\hat{x}(\lambda) \notin dom f(\cdot, y)$. This implies that $\hat{x}(\lambda) \notin dom f(\cdot, y')$ for all $y' \in J$. This renders $w_J$ to be infinity which contradicts our assumption $\sup_{J\in\mcl{J}} w_J < \infty $.

Therefore, now onward, we assume each $w_J = \sup_{\lambda\in M_J} \inf_{x \in \mcl{X}_J} \dotp{\lambda }{ \phi_J(x) }$. To prove the lemma, it suffices to show that $v_J^{\sharp} \le w_J $. Let $\eps>0$ and denote $\mbf{1} \defeq (1, \ldots, 1)$. We shall show that
\begin{align*}
\l( w_J + \eps \r) \mbf{1} \in \phi_J(\mcl{X}_J) + \mbb{R}_{\ge 0}^n .\numberthis\label{eq:wj_in_convex_set}
\end{align*}
Suppose that this is not the case. Since $\phi_J(\mcl{X}_J)+\mbb{R}_{\ge 0}^n$ is a convex set in $\mbb{R}^n $ (see Lemma \ref{lem:lem8.1:aubin:modified}), we may use the hyperplane separation theorem (see Proposition \ref{prop:separation_theorem}), via which there exists $\rho \in \mbb{R}^n$, $\rho \neq 0$, such that
\begin{align*}
\sum_{i=1}^n \rho_i \l( w_J + \eps \r) &=\dotp{\rho}{\l(w_J + \eps\r) \mbf{1}} \\
&\leq \inf_{v \in \phi_J(\mcl{X}_J)+\mbb{R}_{\ge 0}^n} \dotp{\rho}{v}\\
& =\inf_{x \in \mcl{X}_J} \dotp{\rho}{ \phi_J(x)} + \inf_{u \in \mbb{R}_{\ge 0}^n} \dotp{\rho}{u}.
\end{align*}
Then $\inf_{u \in \mbb{R}_{\ge 0}^n} \dotp{\rho}{u}$ is bounded below and consequently, $\rho$ belongs to $\mbb{R}_{\ge 0}^n$ and $\inf_{u \in \mbb{R}_{\ge 0}^n} \dotp{\rho}{u}$ is equal to 0. Since $\rho$ is non-zero, $\sum_{i=1}^n \rho_i$ is strictly positive. We set $\bar{\lambda} =\rho / \sum_{i=1}^n \rho_i \in M_J$ and 
deduce that
\begin{align*}
w_J + \eps &\leq \inf_{x \in \mcl{X}_J } \dotp{\bar{\lambda}} {\phi_J(x)} \\
&\leq \sup_{\substack{\lambda \in M_J}} \inf_{x \in \mcl{X}_J} \dotp{\lambda}{ \phi_J(x) }= w_J.
\end{align*}
This is impossible and thus \eqref{eq:wj_in_convex_set} is established, which implies that there exist $x_{\eps} \in \mcl{X}_J$ and $u_{\eps} \in \mbb{R}_{\ge 0}^n$ such that $\l(w_J+\eps\r) \mathbf{1}=$ $\phi_J\l(x_{\eps}\r)+u_{\eps}$.
From the definition of $\phi_J$, we deduce that
\begin{align*}
\forall i=1, \ldots, n, \quad f\l(x_{\eps}, y_i\r) \leq w_J+\eps,
\end{align*}
and hence 
\begin{align*}
v_J^{\sharp} \leq \max _{i=1, \ldots, n} f\l(x_{\eps}, y_i\r) \leq w_J+\eps.
\end{align*}
We complete the proof of the lemma by letting $\eps$ tend to 0.   
\end{proof}

\begin{lem}\label{lem:lem8.1:aubin:modified}
Consider a function $f : \mcl{X} \times \mcl{Y} \mapsto \mbb{R} \cup \{ \infty \} $ such that $\mcl{X}$ is convex and for each $y \in \mcl{Y}$, $f(\cdot, y)$ is convex. Then, $\phi_J(\mcl{X}_J) + \mbb{R}_{\ge 0}^n$ is a convex set in $\mbb{R}^n$.    
\end{lem}
\begin{proof} 
Take any convex combination $\alpha_1 \l(\phi_J (x_1) + u_1 \r) + \alpha_2 \l(\phi_J(x_2) + u_2\r)$ where $\alpha_1, \alpha_2 \geq 0$, $\alpha_1 + \alpha_2 = 1$, $x_1$ and $x_2$ are in $\mcl{X}_J$, and $u_1$ and $u_2$ are in $\mbb{R}_{\ge 0}^n$. Let $x = \alpha_1 x_1 + \alpha_2 x_2$. For each $y \in J$, the function $f(\cdot, y)$ is convex, therefore $\phi_J(x) \le \alpha_1 \phi_J(x_1) + \alpha_2 \phi_J(x_2) < \infty$ (latter by definition of $\mathcal{X}_J$). Hence, $x \in \mcl{X}_J$. We can write the convex combination in the form $\phi_J (x) + u$ where $u = \alpha_1 u_1 + \alpha_2 u_2 + \alpha_1 \phi_J(x)+\alpha_2 \phi_J(y)-\phi_J(x)$. Note that $u \in \mbb{R}_{\ge 0}^{n}$ because $\phi_J(x) \le \alpha_1 \phi_J(x_1) + \alpha_2 \phi_J(x_2)$. Consequently, $\alpha_1\l(\phi_J\l(x\r)+u_1\r)+\alpha_2\l(\phi_J\l(y\r)+u_2\r)=\phi_J(x)+u$ belongs to $\phi_J(\mcl{X}_J)+\mbb{R}_{\ge 0}^n$.
\end{proof}

\section*{Acknowledgment}
This work was funded by NSF via grants ECCS2038416, EPCN1608361, EARS1516075, CNS1955777, CCF2008130, and CMMI2240981 for V. Subramanian, and grants EARS1516075, CNS1955777, CCF2008130, and CMMI2240981 for N. Khan. The authors would also like to thank Dr. Hsu Kao for helpful discussions. 

\ifCLASSOPTIONcaptionsoff
  \newpage
\fi



\bibliographystyle{ieeetr}
\bibliography{references}

\begin{thebibliography}{10}

\bibitem{bellman57}
R.~Bellman, ``A {M}arkovian decision process,'' {\em Journal of Mathematics and
  Mechanics}, vol.~6, no.~5, pp.~679--684, 1957.

\bibitem{astrom65}
K.~J. Astrom, ``Optimal control of {M}arkov processes with incomplete state
  information,'' {\em Journal of Mathematical Analysis and Applications},
  vol.~10, pp.~174--205, 1965.

\bibitem{howard:dp}
R.~A. Howard, {\em Dynamic Programming and {M}arkov Processes}.
\newblock Cambridge, MA: MIT Press, 1960.

\bibitem{smallwood1973optimal}
R.~D. Smallwood and E.~J. Sondik, ``The optimal control of partially observable
  {M}arkov processes over a finite horizon,'' {\em Operations research},
  vol.~21, no.~5, pp.~1071--1088, 1973.

\bibitem{sondik1978optimal}
E.~J. Sondik, ``The optimal control of partially observable {M}arkov processes
  over the infinite horizon: {D}iscounted costs,'' {\em Operations research},
  vol.~26, no.~2, pp.~282--304, 1978.

\bibitem{kaelbing199899}
L.~P. Kaelbling, M.~L. Littman, and A.~R. Cassandra, ``Planning and acting in
  partially observable stochastic domains,'' {\em Artificial Intelligence},
  vol.~101, no.~1, pp.~99--134, 1998.

\bibitem{sutton98}
R.~S. Sutton and A.~G. Barto, {\em Reinforcement Learning: {A}n Introduction}.
\newblock The MIT Press, second~ed., 2018.

\bibitem{oliehoek16}
F.~A. Oliehoek and C.~Amato, ``A concise introduction to decentralized
  {POMDPs},'' in {\em SpringerBriefs in Intelligent Systems}, 2016.

\bibitem{altman94}
E.~Altman, ``Denumerable constrained {M}arkov decision processes and finite
  approximations,'' {\em Mathematics of Operations Research}, vol.~19, no.~1,
  pp.~169--191, 1994.

\bibitem{altman96}
E.~Altman, ``Constrained {M}arkov decision processes with total cost criteria:
  {O}ccupation measures and primal lp,'' {\em Mathematical Methods of
  Operations Research}, vol.~43, pp.~45--72, Feb 1996.

\bibitem{feinberg94}
E.~A. Feinberg, ``Constrained {Semi-Markov} decision processes with average
  rewards,'' {\em Zeitschrift f{\"u}r Operations Research}, vol.~39,
  pp.~257--288, Oct 1994.

\bibitem{feinberg95}
E.~Feinberg and A.~Shwartz, ``Constrained discounted dynamic programming,''
  {\em Mathematics of Operations Research}, vol.~21, 11 1995.

\bibitem{feinberg96}
E.~A. Feinberg and A.~Shwartz, ``Constrained discounted dynamic programming,''
  {\em Mathematics of Operations Research}, vol.~21, no.~4, pp.~922--945, 1996.

\bibitem{feinberg2000}
E.~A. Feinberg, ``Constrained discounted {M}arkov decision processes and
  hamiltonian cycles,'' {\em Mathematics of Operations Research}, vol.~25,
  no.~1, pp.~130--140, 2000.

\bibitem{feinberg2020}
E.~A. Feinberg, A.~Jaśkiewicz, and A.~S. Nowak, ``Constrained discounted
  {M}arkov decision processes with borel state spaces,'' {\em Automatica},
  vol.~111, p.~108582, 2020.

\bibitem{altman-constrainedMDP}
E.~Altman, {\em Constrained {M}arkov Decision Processes}.
\newblock Chapman and Hall, 1999.

\bibitem{borkar2005AnAA}
V.~S. Borkar, ``An actor-critic algorithm for constrained {M}arkov decision
  processes,'' {\em Syst. Control. Lett.}, vol.~54, pp.~207--213, 2005.

\bibitem{bhatnagar2010AnAA}
S.~Bhatnagar, ``An actor-critic algorithm with function approximation for
  discounted cost constrained {M}arkov decision processes,'' {\em Syst.
  Control. Lett.}, vol.~59, pp.~760--766, 2010.

\bibitem{bhatnagar2012AnOA}
S.~Bhatnagar and K.~Lakshmanan, ``An online actor–critic algorithm with
  function approximation for constrained {M}arkov decision processes,'' {\em
  Journal of Optimization Theory and Applications}, vol.~153, pp.~688 -- 708,
  2012.

\bibitem{Wei2022APM}
H.~Wei, X.~Liu, and L.~Ying, ``A provably-efficient model-free algorithm for
  infinite-horizon average-reward constrained {M}arkov decision processes,'' in
  {\em AAAI Conference on Artificial Intelligence}, 2022.

\bibitem{wei22a-pmlr-v151}
H.~Wei, X.~Liu, and L.~Ying, ``Triple-{Q}: {A} model-free algorithm for
  constrained reinforcement learning with sublinear regret and zero constraint
  violation,'' in {\em Proceedings of The 25th International Conference on
  Artificial Intelligence and Statistics}, vol.~151, pp.~3274--3307, PMLR,
  28--30 Mar 2022.

\bibitem{bura2021}
A.~Bura, A.~HasanzadeZonuzy, D.~Kalathil, S.~Shakkottai, and J.-F. Chamberland,
  ``{DOPE}: {D}oubly optimistic and pessimistic exploration for safe
  reinforcement learning.''
  \url{https://arxiv.org/abs/2112.00885?context=cs.AI}, 2021.

\bibitem{vaswani2022}
S.~Vaswani, L.~Yang, and C.~Szepesvari, ``Near-optimal sample complexity bounds
  for constrained {MDP}s,'' in {\em Advances in Neural Information Processing
  Systems} (A.~H. Oh, A.~Agarwal, D.~Belgrave, and K.~Cho, eds.), 2022.

\bibitem{dongho2011}
D.~Kim, J.~Lee, K.-E. Kim, and P.~Poupart, ``{P}oint-{B}ased {V}alue
  {I}teration for {C}onstrained {POMDP}s,'' in {\em Proceedings of the
  Twenty-Second International Joint Conference on Artificial Intelligence -
  Volume Volume Three}, IJCAI'11, p.~1968–1974, AAAI Press, 2011.

\bibitem{jongmin18}
J.~Lee, G.-h. Kim, P.~Poupart, and K.-E. Kim, ``{Monte-Carlo} tree search for
  constrained {POMDPs},'' in {\em Advances in Neural Information Processing
  Systems}, vol.~31, Curran Associates, Inc., 2018.

\bibitem{undurti2010}
A.~Undurti and J.~P. How, ``An online algorithm for constrained {POMDPs},'' in
  {\em 2010 IEEE International Conference on Robotics and Automation},
  pp.~3966--3973, 2010.

\bibitem{jamgochian2022}
A.~Jamgochian, A.~Corso, and M.~J. Kochenderfer, ``Online planning for
  constrained {POMDPs} with continuous spaces through dual ascent.''
  \url{https://arxiv.org/abs/2212.12154}, 2022.

\bibitem{bernstein00}
D.~S. Bernstein, S.~Zilberstein, and N.~Immerman, ``The complexity of
  decentralized control of markov decision processes,'' in {\em Proceedings of
  the Sixteenth Conference on Uncertainty in Artificial Intelligence},
  p.~32–37, Morgan Kaufmann Publishers Inc., 2000.

\bibitem{witsenhausen1979structure}
H.~S. Witsenhausen, ``On the structure of real-time source coders,'' {\em Bell
  System Technical Journal}, vol.~58, no.~6, pp.~1437--1451, 1979.

\bibitem{witsenhausen1973standard}
H.~S. Witsenhausen, ``A standard form for sequential stochastic control,'' {\em
  Mathematical systems theory}, vol.~7, no.~1, pp.~5--11, 1973.

\bibitem{nayyar13}
A.~Nayyar, A.~Mahajan, and D.~Teneketzis, ``Decentralized stochastic control
  with partial history sharing: {A} common information approach,'' {\em IEEE
  Transactions on Automatic Control}, vol.~58, no.~7, pp.~1644--1658, 2013.

\bibitem{nayyar14}
A.~Nayyar, A.~Mahajan, and D.~Teneketzis, {\em The Common-Information Approach
  to Decentralized Stochastic Control}, pp.~123--156.
\newblock Cham: Springer International Publishing, 2014.

\bibitem{hsu22}
H.~Kao and V.~Subramanian, ``Common information based approximate state
  representations in multi-agent reinforcement learning,'' in {\em Proceedings
  of The 25th International Conference on Artificial Intelligence and
  Statistics}, vol.~151, pp.~6947--6967, PMLR, 28--30 Mar 2022.

\bibitem{gupta17}
J.~K. Gupta, M.~Egorov, and M.~Kochenderfer, ``Cooperative multi-agent control
  using deep reinforcement learning,'' in {\em Autonomous Agents and Multiagent
  Systems}, (Cham), pp.~66--83, Springer International Publishing, 2017.

\bibitem{rashid2020-2}
T.~Rashid, G.~Farquhar, B.~Peng, and S.~Whiteson, ``Weighted {QMIX}:
  {E}xpanding monotonic value function factorisation for deep multi-agent
  reinforcement learning,'' in {\em Advances in Neural Information Processing
  Systems}, vol.~33, pp.~10199--10210, Curran Associates, Inc., 2020.

\bibitem{tulcea49}
C.~T. {Ionescu Tulcea}, ``Mesures dans les espaces produits,'' {\em
  Lincei--Rend. Sc. fis. mat. e nat.}, vol.~7, pp.~208--211, 1949.

\bibitem{kallenberg2002foundations}
O.~Kallenberg, {\em Foundations of modern probability}.
\newblock Probability and its Applications (New York), Springer-Verlag, New
  York, second~ed., 2002.

\bibitem{khan2023cooperative}
N.~Khan and V.~Subramanian, ``Cooperative {M}ulti-{A}gent {C}onstrained
  {P}omdps: {S}trong {D}uality and {P}rimal-{D}ual {R}einforcement {L}earning
  with {A}pproximate {I}nformation {S}tates,'' 2023.

\bibitem{aubin_book_2002}
S.~Wilson and J.~Aubin, {\em Optima and Equilibria: An Introduction to
  Nonlinear Analysis}.
\newblock Graduate Texts in Mathematics, Springer Berlin Heidelberg, 2002.

\end{thebibliography}
%



%

\begin{IEEEbiography}
[{\includegraphics[width=1in, height=1.25in, clip, keepaspectratio]{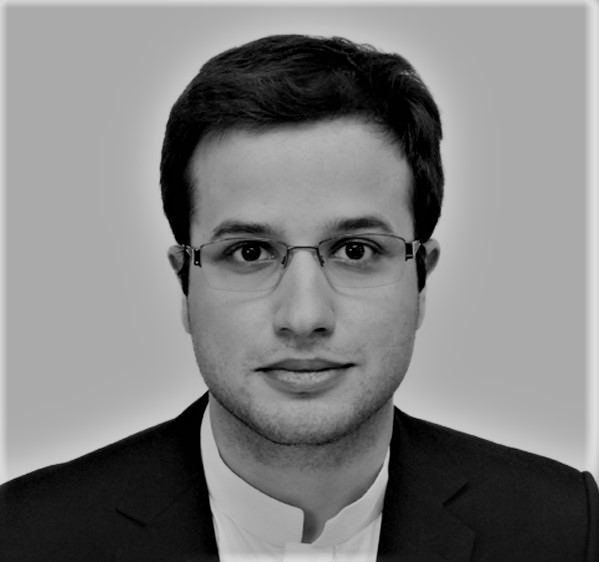}}]{Nouman Khan} (Member, IEEE) is a Ph.D candidate in the department of Electrical Engineering and Computer Science (EECS) at the University of Michigan, Ann Arbor, MI, USA. He received the B.S. degree in Electronic Engineering from the GIK Institute of Engineering Sciences and Technology, Topi, KPK, Pakistan, in 2014 and the M.S. degree in Electrical and Computer Engineering from the University of Michigan, Ann Arbor, MI, USA in 2019. His research interests include stochastic systems and their analysis and control.
\end{IEEEbiography}

\begin{IEEEbiography}[{\includegraphics[width=1in, height=1.25in, clip, keepaspectratio]{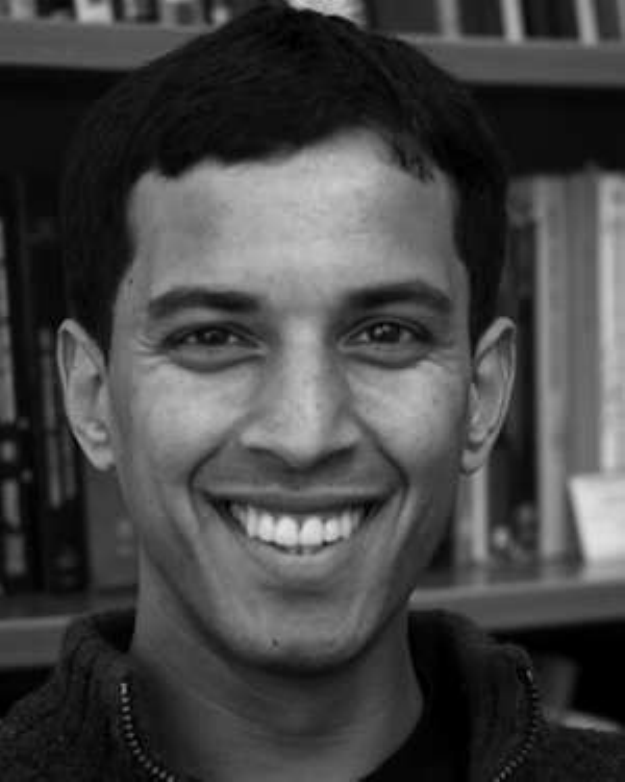}}]{Vijay Subramanian} (Senior Member, IEEE) received the Ph.D. degree in electrical engineering from the University of Illinois at Urbana-Champaign, Champaign, IL, USA, in 1999. He was a Researcher with Motorola Inc., and also with Hamilton Institute, Maynooth, Ireland, for a few years following which he was a Research Faculty with the Electrical Engineering and Computer Science (EECS) Department, Northwestern University, Evanston, IL, USA. In 2014, he joined the University of Michigan, Ann Arbor, MI, USA, where he is currently an Associate Professor with the EECS Department. His research interests are in stochastic analysis, random graphs, game theory, and mechanism design with applications to social, as well as economic and technological networks. 
\end{IEEEbiography}
\end{document}